\documentclass{article}

\usepackage{latexsym}

\usepackage[latin1]{inputenc}
\usepackage{amsfonts,amsmath,amssymb}
\usepackage{a4wide}

\newtheorem{conj}{\textbf{Conjecture}}[section]
\newtheorem{thm}[conj]{\textbf{Theorem}}
\newtheorem{lem}[conj]{\textbf{Lemma}}
\newtheorem{definition}[conj]{\textbf{Definition}}
\newtheorem{propo}[conj]{\textbf{Proposition}}

\newtheorem{corol}[conj]{Corollary}

\newenvironment{proof}{\paragraph{Proof}}{}
\newtheorem{rem}[conj]{Remark}

\newenvironment{theorem}{\begin{thm}\sl}{\end{thm}}
\newenvironment{rtheorem}[1]{\begin{thm}[#1]\sl}{\end{thm}}

\newenvironment{remark}{\begin{rem}\rm}{\end{rem}}

\newenvironment{lemma}{\begin{lem}\sl}{\end{lem}}

\newenvironment{corollary}{\begin{corol}\sl}{\end{corol}}

\newenvironment{proposition}{\begin{propo}\sl}{\end{propo}}

\newcommand\aheight{{\mathrm{h_a}}}
\newcommand\pheight{{\mathrm{h_p}}}

\newcommand\vheight{{\mathrm h}}

\newcommand\Q{\mathbb{Q}}
\newcommand\C{\mathbb{C}}
\newcommand\Z{\mathbb{Z}}
\newcommand\R{\mathbb{R}}
\newcommand\barq{{\bar{\mathbb{Q}}}}
\newcommand\bark{{\bar{\mathbb{K}}}}
\newcommand\barr{{\bar R}}
\newcommand\barv{{\bar v}}

\newcommand\PP{{\mathbb P}}
\newcommand\K{{\mathbb K}}
\renewcommand\L{{\mathbb L}}
\newcommand\M{{\mathbb M}}

\newcommand\DD{{\mathcal D}}

\newcommand\norm{{\mathcal N}}
\newcommand\OO{{\mathcal O}}
\newcommand\CC{{\mathcal C}}
\renewcommand\AA{{\mathcal{A}}}
\newcommand\QQ{{\mathcal{Q}}}

\newcommand\ord{\mathrm{ord}}

\newcommand\ram{\mathrm{Ram}}

\newcommand\Den{\mathrm{Den}}
\newcommand\Num{\mathrm{Num}}

\newcommand\Ess{\mathrm{Ess}\,}

\newcommand\ual{{\underline\alpha}}

\newcommand\ua{{\underline a}}
\newcommand\uzero{{\underline 0}}

\newcommand\gerd{{\mathfrak{d}}}

\newcommand\hatf{{\widehat{F}}}
\newcommand\hatr{{\widehat{R}}}

\newcommand\qed{\hfill$\square$}

\newcommand\tilcc{{\widetilde \CC}}
\newcommand\tile{{\widetilde e}}
\newcommand\tilf{{\tilde f}}
\newcommand\tilg{{\widetilde g}}
\newcommand\tiln{{\widetilde n}}
\newcommand\tilp{{\widetilde P}}
\newcommand\tilP{{\widetilde P}}
\newcommand\tilm{{\widetilde m}}
\newcommand\tilq{{\widetilde Q}}
\newcommand\tilr{{\widetilde R}}
\newcommand\tilt{{\widetilde T}}
\newcommand\tilw{{\widetilde w}}
\newcommand\tilz{{\widetilde z}}
\newcommand\tilZ{{\widetilde Z}}
\newcommand\tilaa{{\widetilde \AA}}
\newcommand\tilqq{{\widetilde \QQ}}
\newcommand\tilzeta{{\tilde \zeta}}
\newcommand\tilom{{\widetilde \Omega}}
\newcommand\tilV{{\widetilde V}}
\newcommand\tily{{\widetilde y}}
\newcommand\tilY{{\widetilde Y}}
\newcommand\tilN{{\widetilde N}}
\newcommand\tilomega{{\widetilde \omega}}

\newcommand\genus{{\mathbf g}}
\newcommand\tilgen{{\widetilde \genus}}
\newcommand\tilam{{\widetilde \Lambda}}

\makeatletter

\renewcommand*\l@section[2]{%
  \ifnum \c@tocdepth >\z@
    \addpenalty\@secpenalty
    \addvspace{0.3em \@plus\p@}%
    \setlength\@tempdima{1.5em}%
    \begingroup
      \parindent \z@ \rightskip \@pnumwidth
      \parfillskip -\@pnumwidth
      \leavevmode \bfseries
      \advance\leftskip\@tempdima
      \hskip -\leftskip
      #1\nobreak\hfil \nobreak\hb@xt@\@pnumwidth{\hss #2}\par
    \endgroup
  \fi}
  
 \makeatother

\author{%

Yuri Bilu\footnotemark \; (Bordeaux), \\
Marco Strambi (Livorno), \\

Andrea Surroca\footnotemark \; (Basel)}

\title{Quantitative Chevalley-Weil Theorem for Curves}
\date{\today}

1

\begin{document}

\hfuzz 4pt

\maketitle

\footnotetext[1]{Supported by the ANR project HAMOT, and  \textit{Ambizione} fund PZ00P2\_121962 of the Swiss National Foundation}
\footnotetext[2]{Supported by the Marie Curie IEF 025499 of the European Community and the Ambizione Fund PZ00P2\_121962 of the Swiss National Science Foundation}

\setcounter{footnote}2

\begin{abstract}
The classical Chevalley-Weil theorem asserts that for an \'etale covering  of projective varieties over a number field~$\K$, the discriminant of the field of definition of the fiber over a $\K$-rational point is uniformly bounded. 
We obtain a fully explicit version of this theorem   in dimension~$1$.
\end{abstract}


\bigskip


{\footnotesize
\tableofcontents}

\section{Introduction}
The Chevalley-Weil theorem is one of the most basic principles of the Diophantine analysis. Already Diophantus of Alexandria routinely used reasoning of the kind ``if~$a$ and~$b$ are `almost' co-prime integers and $ab$ is a square, then each of~$a$ and~$b$ is `almost' a square''.  The Chevalley-Weil~\cite{CW32,We35} theorem provides a general set-up for this kind of arguments. 

\begin{rtheorem}{Chevalley-Weil}
Let ${\tilV\stackrel\phi\to V}$ be a finite \'etale covering of normal projective varieties, defined over a number field~$\K$. Then there exists a non-zero integer~$T$ such that for any   ${P\in V(\K)}$ and ${\tilP \in \tilV(\bar \K)}$ such that ${\phi(\tilP)=P}$, the relative discriminant of ${\K(\tilP)/\K}$ divides~$T$. 
\end{rtheorem}

There is also a similar statement for coverings of affine varieties and integral points. See \cite[Section~2.8]{La83} or \cite[Section~4.2]{Se97} for more details.

The Chevalley-Weil theorem is indispensable in the Diophantine analysis, because it reduces a Diophantine problem on the variety~$V$ to that on the covering variety~$\tilV$, which can often be simpler to deal. In particular, the Chevalley-Weil theorem is used, albeit implicitly, in the proofs of the great finiteness theorems of Mordell-Weil, Siegel and Faltings. 

In view of all this, a quantitative version of the Chevalley-Weil theorem, at least in dimension~$1$, would be useful to have. One such version appears in Chapter~4 of~\cite{Bi93}, but it is not explicit in all parameters; neither is the version  recently suggested by Draziotis and Poulakis \cite{DP09,DP10}, who also make some other restrictive assumptions (see Remark~\ref{rdrap} below). 

In the present article we present a version of the Chevalley-Weil theorem in dimension~$1$, which is explicit in all parameters and considerably sharper than the previous versions. Our approach is  different from that of \cite{DP09,DP10} and goes back to~\cite{Bi93,Bi97}. 

To state our principal results, we have to introduce some notation. Let~$\K$ be a number field,~$\CC$
an absolutely irreducible smooth projective curve~$\CC$ defined over~$\K$, and ${x\in \K(\CC)}$ a non-constant $\K$-rational function on~$\CC$. We also fix a covering ${\tilcc \stackrel\phi\to\CC}$ of~$\CC$ by another smooth irreducible projective curve~$\tilcc$; we assume that both~$\tilcc$ and the covering~$\phi$ are defined over~$\K$. We consider $\K(\CC)$ as a subfield of $\K(\tilcc)$; in particular, we identify the functions ${x\in \K(\CC)}$ and ${x\circ\phi\in \K(\tilcc)}$.

We also fix
one more rational function ${y\in \K(\CC)}$ such that ${\K(\CC)=\K(x,y)}$ (existence of such~$y$ follows from the primitive element theorem). Let ${f(X,Y)\in \K[X,Y]}$ be the $\K$-irreducible polynomial such that ${f(x,y)=0}$ (it is well-defined up to a constant factor). Since~$\CC$ is absolutely irreducible, so is the polynomial $f(X,Y)$. We put 
${m=\deg_X f}$ and ${n=\deg_Yf}$.

Similarly, we fix a function ${\tily \in \K(\tilcc)}$ such that ${K(\tilcc)=\K(x,\tily )}$. We let ${\tilf(X,\tilY ) \in \K[X,\tilY ]}$ be an irreducible polynomial  such that ${\tilf(x,\tily )=0}$. We put 
${\tilm=\deg_X \tilf}$ and ${\tiln=\deg_Y\tilf}$. 
We denote by~$\nu$ the degree of the covering~$\phi$, so that ${\tiln=n\nu}$. 

\begin{remark}
Equations ${f(X,Y)=0}$ and ${\tilf(X,\tilY)=0}$ define affine plane models of our curves~$\CC$ and~$\tilcc$; \textsl{we do not assume these models non-singular}. 
\end{remark}

In the sequel, $\pheight(\cdot)$ and $\aheight(\cdot)$ denote the projective and the affine absolute logarithmic heights, respectively, see Section~\ref{snota} for the definitions. We also define normalized logarithmic discriminant $\partial_{\L/\K}$ and  the height $\vheight(S)$ of a finite set of places~$S$ as 
$$
\partial_{\L/\K}=\frac{\log\norm_{\K/\Q}\DD_{\L/\K}}{[\L:\Q]}, \qquad \vheight(S) =\frac{\sum_{v\in S}\log\norm_{\K/\Q}(v)}{[\K:\Q]};
$$
see Section~\ref{snota} for the details. 

Put
\begin{equation}
\label{eom}
\begin{gathered}
\Omega= mn^2 \bigl(\pheight(f)+2m+2n\bigr),\qquad
\tilom= \tilm\tiln^2 \bigl(\pheight(\tilf)+2\tilm+2\tiln\bigr),\\
\Upsilon=2\tiln\bigl(\tilm\pheight(f)+m\pheight(\tilf)\bigr). 
\end{gathered}
\end{equation}

\begin{rtheorem}{``projective'' Chevalley-Weil theorem}
\label{tproj}
In the above set-up, assume that the covering ${\tilcc\stackrel\phi\to \CC}$ is unramified. Then for every ${P\in \CC(\bark)}$ and ${\tilp\in \tilcc(\bark)}$ such that ${\phi(\tilp)=P}$ we have
$$
\partial_{\K(\tilp)/\K(P)}\le 400(\Omega+\tilom) +2\Upsilon +6m\tiln^2. 
$$
\end{rtheorem}

\begin{remark}
\label{rdrap}
Draziotis and Poulakis \cite[Theorem~1.1]{DP10}, assume that~$\CC$ is a non-singular plane curve (which is quite restrictive) and that ${P\in \CC(\K)}$. Their set-up is slightly different, and the two estimates cannot be compared directly. But it would be safe to say that their estimate is not sharper than 
$$
\partial_{\K(\tilp)/\K(P)}\le cN^{30}\tilN^{13}\left(\pheight(f)+ \pheight(\tilf)\right) + C,
$$
where
${N=\deg f}$,  ${\tilN=\deg \tilf}$, the constant~$c$ is absolute and~$C$ depends of~$N$,~$\tilN$ and the degree $[\K:\Q]$. 
\end{remark}

Now let~$S$ be a finite set of places of~$\K$, including all the archimedean places. A point ${P\in \CC(\bark)}$ will be called \textsl{$S$-integral} if for any ${v\in M_\K\smallsetminus S}$ and any extension~$\barv$ of~$v$ to~$\bark$ we have ${|x(P)|_\barv\le 1}$. 

\begin{rtheorem}{``affine'' Chevalley-Weil theorem}
\label{taff}
In the above set-up, assume that the covering ${\tilcc\stackrel\phi\to \CC}$ is unramified outside the poles of~$x$. Then for every $S$-integral point ${P\in \CC(\bark)}$ and ${\tilp\in \tilcc(\bark)}$ such that ${\phi(\tilp)=P}$ we have
\begin{equation}
\label{eaff}
\partial_{\K(\tilp)/\K(P)}\le   300(\Omega+\tilom) +\Upsilon+3m\tiln^2+\vheight(S). 
\end{equation}
\end{rtheorem}

Again, Draziotis and Poulakis \cite[Theorem~1.1]{DP09} obtain a less sharp result under more restrictive assumptions.

It might be also useful to have a statement free of the defining equations of the curves~$\CC$ and~$\tilcc$. Using the result of~\cite{BS10}, we obtain versions of Theorems~\ref{tproj} and~\ref{taff}, which depend only on the degrees and the ramification points of our curves over~$\PP^1$. For a finite set ${A\subset \PP^1(\bark)}$ we define $\aheight(A)$ as the affine height of the vector whose coordinates are the finite elements of~$A$.

\begin{theorem}
\label{tminim}
Let~$A$ be a finite subset of $\PP^1(\bark)$ such that the covering ${\CC\stackrel x\to\PP^1}$ is unramified outside~$A$. Put 
$$
\delta=[\K(A):\K], \qquad \tilgen =\genus(\tilcc), \qquad \Lambda= \bigl((\tilgen+1)\tiln\bigr)^{25(\tilgen+1)\tiln} +2(\delta-1).
$$ 
\begin{enumerate}
\item
\label{iproj}
Assume that the covering ${\phi:\tilcc\to \CC}$ is unramified. Then for every ${P\in \CC(\bark)}$ and ${\tilp\in \tilcc(\bark)}$ such that ${\phi(\tilp)=P}$ we have
$$
\partial_{\K(\tilp)/\K(P)}  \le \Lambda\bigl(\aheight(A)+1\bigr).
$$
\item
Assume that the covering ${\phi:\tilcc\to \CC}$ is unramified outside the poles of~$x$, and let~$S$ be as above. Then for every $S$-integral point ${P\in \CC(\bark)}$ and ${\tilp\in \tilcc(\bark)}$ such that ${\phi(\tilp)=P}$ we have
$$
\partial_{\K(\tilp)/\K(P)} \le \vheight(S)+  \Lambda\bigl(\aheight(A)+1\bigr). 
$$
\end{enumerate}
\end{theorem}

\paragraph{Acknowledgments}
The authors thank Carlo Gasbarri for useful discussions. 
Yuri Bilu thanks the Mathematical Institute of the University of Basel for hospitality in late 2011, when a substantial part of this work was done. 

We thank the anonymous referee for the encouraging report, and for detecting an inaccuracy in the original version.

\section{Preliminaries}
\label{snota}

Let~$\K$ be any number field  and let
${M_\K=M_\K^0\cup M_\K^\infty}$ be the set of its places, with $M_\K^0$ and $M_\K^\infty$ denoting the sets of finite and infinite places, respectively. For every place $v\in M_\K$  we  normalize the corresponding valuation $|\cdot|_v$ so
that its restriction to~$\Q$ is the standard infinite or $p$-adic 
valuation.
Also,  we let $\K_{v}$ be the $v$-adic completion
of $\K$, (in particular,~$\K_v$ is  $\R$ or $\C$ when~$v$ is infinite).

\paragraph{Heights}
For a vector  ${\ual =(\alpha_1, \ldots, \alpha_N) \in \barq^N}$ 
we define, as usual, the \textsl{absolute logarithmic projective height} and  \textsl{absolute logarithmic affine height} (in the sequel simply \textsl{projective} and \textsl{affine heights}) by\footnote{In the definition of the projective height we assume that at least one coordinate of~$\ual$ is non-zero.}
\begin{equation}
\label{ehei}
\pheight(\ual)=\frac{1}{[\K:\Q]}\sum_{v \in M_{\K}} [\K_v:\Q_v] \log\|\ual\|_v ,\qquad
\aheight(\ual)=\frac{1}{[\K:\Q]}\sum_{v \in M_{\K}} [\K_v:\Q_v] \log^+\|\ual\|_v, 
\end{equation}
where~$\K$ is any number field containing the coordinates of~$\ual$,
$$
\|\ual\|_v= \max \{|\alpha_0|_{v}, \ldots, |\alpha_N|_{v}\}
$$
and ${\log^+=\max\{\log,0\}}$. With our choice of normalizations, the right-hand sides in~\eqref{ehei} are independent  of the choice of the field~$\K$. For a polynomial~$f$ with algebraic coefficients we denote by  ${\pheight(f)}$ and by ${\aheight(f)}$ the projective height and the affine height of the vector of its coefficients respectively, ordered somehow.

\paragraph{Logarithmic discriminant}
Given an extension ${\L/\K}$ of number fields, we denote by $\partial_{\L/\K}$ the \textsl{normalized logarithmic relative discriminant}:
$$
\partial_{\L/\K}=\frac{\log\norm_{\K/\Q}\DD_{\L/\K}}{[\L:\Q]},
$$
where $\DD_{\L/\K}$ is the usual relative discriminant and $\norm_{\K/\Q}$ is the norm map.
The properties of this quantity are summarized in the following proposition. 
\begin{proposition}
\begin{enumerate}
\item
(additivity in towers)\quad  If ${\K\subset\L\subset \M}$ is a tower of number fields, then  
${\partial_{\M/\K}= \partial_{\L/\K}+\partial_{\M/\L}}$.

\item (base extension)\quad
If~$\K'$ is a finite extension of~$\K$ and ${\L'=\L\K'}$ then
${\partial_{\L'/\K'}\le \partial_{\L/\K}}$.

\item
(triangle inequality)\quad If~$\L_1$ and~$\L_2$ are two extensions of~$\K$, then 
${\partial_{\L_1\L_2/\K}\le \partial_{\L_1/\K}+  \partial_{\L_2/\K}}$.

\end{enumerate}
\end{proposition}
These properties will be used without special reference. 

\paragraph{Height of a set of places}
Given a number field~$\K$ and finite set of places ${S\subset M_\K}$, we define the \textsl{absolute logarithmic height} of this set as 
$$
\vheight(S) =\frac{\sum_{v\in S}\log\norm_{\K/\Q}(v)}{[\K:\Q]}, 
$$
where the norm $\norm_{\K/\Q}(v)$ of the place~$v$ is the norm of the corresponding prime ideal if~$v$ is finite, and is set to be~$1$ when~$v$ is infinite. 
The properties of this height are summarized in the following proposition.

\begin{proposition}
\label{pvhei}
\begin{enumerate}
\item
\label{isl}
(field extension)\quad Let~$\L$ be an extension of~$\K$ and~$S_\L$ the set of extensions of the places from~$S$ to~$\L$. Assume that  no place from~$S$ ramifies in~$\L$. Then ${\vheight(S)=\vheight(S_\L)}$. Without this assumption we have the inequalities  ${\vheight(S_\L)\le\vheight(S)\le [\L:\K]\vheight(S_\L)}$.

\item (denominators and numerators)\quad
\label{idenum}
For ${\ual\in \K^N}$ let the sets  $\Den(\ual)$ and $\Num(\ual)$ consist of all ${v\in M_\K}$  such that ${\|\ual\|_{v}>1}$, respectively, ${\|\ual\|_{v}<1}$.  Then
\begin{align*}
\vheight\bigl(\Den_\K(\ual)\bigr)&\le \aheight(\ual),\\
\vheight\bigl(\Num_\K(\ual)\bigr)&\le \bigl(\aheight(\ual)-\pheight(\ual)\bigr) \qquad (\ual\ne \uzero).
\end{align*}
In particular, for ${\alpha \in \K^\ast}$ we have
${\vheight\bigl(\Num_\K(\alpha)\bigr)\le \aheight(\alpha)}$. \qed

\end{enumerate}

\end{proposition}
This will also be used without special reference. 

\paragraph{Sums over primes}
We shall systematically use the following estimates from~\cite{RS62}:
\begin{align}
\label{epix}
\sum_{p\le x}1&\le 1.26\frac x{\log x},\\
\label{ethetax}
\sum_{p\le x}\log p&\le 1.02 x.
\end{align}
See~\cite{RS62}, Corollary~1 of Theorem~2 for~(\ref{epix}) and  Theorem~9 for~(\ref{ethetax}).

\section{Auxiliary Material}

In this section we collect miscellaneous facts, mostly elementary and/or well-known, to be used in the article. 

\subsection{Integral Elements}

In this subsection~$R$ is an integrally closed integral domain and~$\K$ its quotient field.

\begin{lemma}
\label{ldisc}
Let~$\L$ be a finite separable extension of~$\K$ of degree~$n$ and~$\barr$ the integral closure of~$R$ in~$\L$.  Let  ${\omega_1,\ldots,\omega_n \in \barr}$ form a base of~$\L$ over~$\K$. 
We denote by~$\Delta$ the discriminant of this basis:
${\Delta=\left(\det\left[\sigma_i(\omega_j)\right]_{ij}\right)^2}$,
where ${\sigma_1,\ldots,\sigma_n:\L\hookrightarrow\bar\K}$ are the distinct embeddings of~$\L$ into~$\bar\K$.
Then 
${\barr \subset \Delta^{-1}(R\omega_1 + \cdots + R\omega_n)}$.
\end{lemma}

\begin{proof}
This is standard. Write ${\beta \in \barr}$ as ${\beta=a_1\omega_1+\cdots+a_n\omega_n}$ with ${a_i\in \K}$. Solving the system of linear equations 
$$
\sigma_i(\beta) = a_1\sigma_i(\omega_1)+\cdots+a_n\sigma_i(\omega_n)  \qquad (i=1, \ldots, n)
$$
using the Kramer rule, we find that the numbers $\Delta a_i$ are integral over~$R$. Since~$R$ is integrally closed, we have ${\Delta a_i\in R}$. \qed
\end{proof}

\begin{corollary}
\label{cint}
Let 
${f(T)=f_0T^n + f_1 T^{n-1} + \cdots + f_n \in R[T]}$ 
be a $\K$-irreducible polynomial, and  ${\alpha\in\bark}$ one of its roots. Let~$\barr$ be the integral closure of~$R$ in $\K(\alpha)$. Then 
${\barr \subset \Delta(f)^{-1}R[\alpha]}$.
where $\Delta(f)$ is the discriminant of~$f$.
\end{corollary}

\begin{proof}
It is well-known that the quantities
\begin{align*}
 \omega_1=1, \quad
 \omega_2=f_0\alpha, \quad
 \omega_3=f_0\alpha^2 + f_1 \alpha, \quad \ldots\quad \quad
\omega_n= f_0\alpha^{n-1}+f_1 \alpha^{n-2} + \cdots + f_{n-2}\alpha
\end{align*}
are integral over~$R$; see, for example, \cite[page~183]{Sc91}.
Applying Lemma~\ref{ldisc} to the basis $\omega_1,\ldots,\omega_n$, we  complete the proof. \qed
\end{proof}

\subsection{Local Lemmas}
\label{ssloc}
In this subsection~$\K$ is a field of characteristic~$0$  supplied with a discrete valuation~$v$. We denote by~$\OO_v$ the local ring of~$v$.

The proof of the following lemma is a simple exercise  left to the reader.

\begin{lemma}
\label{lramm}
Assume that~$K$ is complete, and let~$\pi$ be a primitive element of~$\K$. 
\begin{enumerate}
\item
\label{irad}
Let ${\alpha\in \K^\times}$. For a positive integer~$e$  not divisible by the characteristic of the residue field and any choice of the root ${\alpha^{1/e}}$ the ramification index of ${\K(\alpha^{1/e})/\K}$ is ${e/\gcd(e,\ord_\pi\alpha)}$.

\item
\label{icompos}
Let~$\L_1$ and~$\L_2$ be finite extensions of~$\K$ (inside some algebraic closure of~$\K$) of ramification~$e_1$ and~$e_2$ respectively. Assume that none of~$e_1$,~$e_2$ is divisible by the characteristic of the residue field. Then the ramification of $\L_1\L_2/\K$ is $\mathrm{lcm}(e_1,e_2)$. \qed

\end{enumerate}
\end{lemma}

We say that a polynomial ${F(X)\in \K[X]}$ is \emph{$v$-monic} if its leading coefficient is a $v$-adic unit\footnote{We say that~$\alpha$ is a $v$-adic unit if ${|\alpha|_v=1}$.}. 
\begin{lemma}
\label{lram}
Let ${F(X)\in \OO_v[X]}$ be a $v$-monic polynomial,  and let ${\eta\in \bark}$ be a root of~$F$. (We do not assume~$F$ to be the minimal polynomial of~$\eta$ over~$\K$, because we do not assume it $\K$-irreducible.) 
Then ${|F'(\eta)|_v<1}$ for any extension of~$v$ to $\K(\eta)$ ramified over~$\K$. 
\end{lemma}

\paragraph{Proof}
Fix an extension of~$v$ to $\K(\eta)$ ramified over~$\K$. Replacing~$\K$ by~$\K_v$ and $\K(\eta)$ by $\K(\eta)_v$, we 
may assume that~$\K$ is $v$-complete. Let ${\gerd=\gerd_{\K(\eta)/\K}}$ be the different of the extension $\K(\eta)/\K$. Since~$v$ ramifies in $\K(\eta)$, the different is a non-trivial ideal of~$\OO_v$.  

Since~$\eta$ is a root of a $v$-monic polynomial, it is integral over~$\OO_v$. Let ${G(X)\in \OO_v[X]}$ be the minimal polynomial of~$\eta$. Then the different~$\gerd$  divides $G'(\eta)$, which  implies that ${|G'(\eta)|_v<1}$. 

Write ${F(X)=G(X)H(X)}$. By the Gauss lemma, ${H(X)\in \OO_v[X]}$. Since ${F'(\eta)=G'(\eta)H(\eta)}$, we obtain 
${|F'(\eta)|_v\le|G'(\eta)|_v<1}$, as wanted. \qed

\bigskip

Given a polynomial $F(X)$ over some field of characteristic~$0$, we define by $\hatf(X)$ the \textsl{radical} of~$F$, that is, the  separable polynomial, having the same roots  and the same leading coefficient as~$F$:
$$
\hatf(X)=f_0\prod_{F(\alpha)=0}(X-\alpha),
$$
where~$f_0$ is the leading coefficient of~$F$ and the product runs over the \textbf{distinct} roots of~$F$ (in an algebraic closure of the base field).

\begin{lemma}
\label{lfcap}
Assume that ${F(X)\in \OO_v[X]}$. Then the radical $\hatf(X)$  is in $\OO_v[X]$ as well. Also, if ${|F(\xi)|_v<1}$ for some ${\xi\in \OO_v}$, then we have ${|\hatf(\xi)|_v<1}$ as well. 
\end{lemma}

\begin{proof}
Let 
${F(X)=P_1(X)^{\alpha_1} \cdots P_k(X)^{\alpha_k}}$
be the irreducible factorization of~$F$ in ${\K[X]}$. The Gauss Lemma implies that we can choose ${P_i(X) \in \OO_v[X]}$ for ${i=1,\ldots,k}$.
Since the characteristic of~$\K$ is~$0$, every~$P_i$ is separable. Obviously, the leading coefficient of the separable polynomial ${P_1(X)\cdots P_k(X)}$ divides that of~$F(X)$ in the ring~$\OO_v$. Hence
${\hatf(X)=\gamma P_1(X)\cdots P_k(X)}$
with some ${\gamma \in \OO_v}$, which proves the first part of the lemma. 
The second part is obvious: if ${|F(\xi)|_v<1}$ then ${|P_i(\xi)|_v<1}$ for some~$i$, which implies ${|\hatf(\xi)|_v<1}$. \qed
\end{proof}

\begin{lemma}
\label{lhens}
Let ${F(X)\in \OO_v[X]}$ and  ${\xi\in \OO_v}$ satisfy
${|F(\xi)|_v<1}$ and  ${|F'(\xi)|_v=1}$. Let~$\barv$ be an extension of~$v$ to~$\bark$. Then there exists exactly one root ${\alpha\in \bark}$ of~$F$ such that  ${|\xi-\alpha|_\barv<1}$. 
\end{lemma}

\begin{proof}
This is a consequence of Hensel's lemma. Extending~$\K$, we may assume that it contains all the roots of~$F$. Hensel's lemma implies that there is exactly one root~$\alpha$ in  the $v$-adic completion of~$\K$ with the required property. This root must belong to~$\K$. \qed
\end{proof}

\begin{lemma}
\label{loneoverf}
Let ${F(X), G(X)\in \OO_v[X]}$ and  ${\alpha, \xi\in \OO_v}$ satisfy 
$$
F(X)=(X-\alpha)^m G(X),\qquad G(\alpha) \ne 0, \qquad |\xi-\alpha|_v<|G(\alpha)|_v
$$
with some non-negative integer~$m$. Expand the rational function ${F(X)^{-1}}$ into the Laurent series at~$\alpha$. Then this series converges at ${X=\xi}$. 
\end{lemma}

\begin{proof}
Substituting ${X\mapsto\alpha+X}$, we may assume ${\alpha=0}$, in which case the statement becomes obvious. \qed
\end{proof}

\subsection{Heights}

Recall that, for a polynomial~$f$ with algebraic coefficients, we denote by  ${\pheight(f)}$ and by ${\aheight(f)}$, respectively, the projective height and the affine height of the vector of its coefficients ordered somehow. More generally, the height ${\aheight(f_1, \ldots, f_s)}$ of a finite system of polynomials is, by definition, the affine height of the vector formed of all the non-zero coefficients of all these polynomials.

\begin{lemma}
\label{lsombrapol}
Let ${f_1,\ldots,f_s}$ be polynomials in ${\bar\Q[X_1,\ldots,X_r]}$ and put 
$$
\textstyle
N=\max\{\deg f_1, \ldots, \deg f_s\}, \qquad 
h=\aheight(f_1, \ldots,f_s).
$$ 
Let also~$g$ be a polynomial in $\bar\Q[Y_1,\ldots,Y_s]$. Then
\begin{enumerate}
\item 
\label{iprod}
$\aheight\left(\prod_{i=1}^s f_i\right)\le \sum_{i=1}^s \aheight\left(f_i\right) + \log (r+1)\sum_{i=1}^{s-1}\deg f_i$,
\item
\label{iprodinv}
$\pheight\left(\prod_{i=1}^s f_i\right)\ge \sum_{i=1}^s \pheight\left(f_i\right) - \sum_{i=1}^s\deg f_i$,
\item 
\label{ig}
$\aheight\bigl(g\left(f_1,\ldots,f_s\right)\bigr) \le \aheight(g) + \bigl(h+\log(s+1)+N\log(r+1)\bigr) \deg g$. 
\end{enumerate}
\end{lemma}
Notice that we use the projective height in item~\ref{iprodinv}, and the affine height in the other items.

\begin{proof}	
Item~\ref{iprodinv} is the famous Gelfond inequality, see, for instance, Proposition~B.7.3 in~\cite{HS00}. The rest is an immediate consequence of Lemma~1.2 from~\cite{KPS01}. \qed
\end{proof}

\begin{remark}
\label{rts}
If in item~\ref{ig}  we make substitution ${Y_i=f_i}$ only for a part of the indeterminates~$Y_i$, say, for~$t$ of them, where ${t\le s}$, then we may replace ${\log (s+1)}$ by ${\log(t+1)}$, and $\deg g$ by the degree with respect to these indeterminates:
$$
\aheight\bigl(g\left(f_1,\ldots,f_t, Y_{t+1},\ldots, Y_s\right)\bigr) \le \aheight(g) + \bigl(h+\log(t+1)+N\log(r+1)\bigr) \deg_{Y_1, \ldots, Y_t} g.
$$
\end{remark}

\begin{remark}
\label{rheiroots}
When all the~$f_i$ are just linear polynomials in one variable, item~\ref{iprodinv} can be refined as follows:
let ${F(X)}$ be a polynomial of degree~$\rho$, and ${\beta_1,\ldots,\beta_\rho}$ are its roots (counted with multiplicities); then
$$
\aheight(\beta_1)+\cdots+\aheight(\beta_\rho) \leq \pheight(F)+\log(\rho+1).
$$
This is a classical result of Mahler, see, for instance, \cite[Lemma 3]{Sc90}. 
\end{remark}

\begin{corollary}
\label{cdivi}
Let~$f$ and~$g$ be polynomials with algebraic coefficients such that~$f$ divides~$g$. Let also~$a$ be a non-zero coefficient of~$f$. Then 
\begin{enumerate}
\item\label{ihp}
$\pheight(f)\le \pheight(g)+\deg g$,
\item
\label{iha}
$\aheight(f)\le \pheight(g)+\aheight(a)+\deg g$.
\end{enumerate}
\end{corollary}

\begin{proof}
Item~\ref{ihp}  is a direct consequence of item~\ref{iprodinv} of Lemma~\ref{lsombrapol}.
For item~\ref{iha} remark that one of the coefficients of~$f/a$  is~$1$, which implies that
$$
\aheight(f/a)=\pheight(f/a)= \pheight(f)\le \pheight(g)+\deg g.
$$ 
Since
${\aheight(f)\le\aheight(a)+\aheight(f/a)}$, the result follows.\qed
\end{proof}

\begin{corollary}
\label{cfalpha}
Let~$\alpha$ be an algebraic number and ${f\in\barq[X,Y]}$ be a polynomial with algebraic coefficients, let also ${f^{(\alpha)}(X,Y)=f(X+\alpha,Y)}$ then
$$
\aheight(f^{(\alpha)}) \leq \aheight(f)+m\aheight(\alpha)+2m\log2,
$$
where ${m=\deg_X f}$. 
\end{corollary}

\begin{proof}
This  is a direct application of item~\ref{ig} of Lemma~\ref{lsombrapol}, together with Remark~\ref{rts}. \qed
\end{proof}

\medskip
In one special case item~\ref{ig} of Lemma~\ref{lsombrapol} can be refined. 

\begin{lemma}
\label{lsombradet}
Let 
$$
F_{ij}(X)\in \bar\Q[X]  \qquad (i,j=1,\ldots, s)
$$
be polynomials of degree bounded by~$\mu$ and of affine height bounded by~$h$; then 
$$
\aheight\left(\det(F_{ij})\right) \le sh+s(\log s+\mu\log2).
$$
\end{lemma}
For the proof see~\cite{KPS01}, end of Section~1.1.1. 

\medskip

We also need an estimate for both the affine and the projective height of the $Y$-resultant $R_f(X)$ of a polynomial ${f(X,Y)\in \bar\Q[X,Y]}$ 
and its $Y$-derivative $f'_Y$, in terms of the affine (respectively, projective) height of~$f$. 
\begin{lemma}
\label{lheires}
Let ${f(X,Y)\in \bar\Q[X,Y]}$ be of $X$-degree~$m$ and $Y$-degree~$n$. Then
\begin{align}
\label{eaffres}
\aheight(R_f)&\le (2n-1)\aheight(f)+(2n-1)\left(\log(2n^2)+m\log 2\right),\\
\label{eprojres}
\pheight(R_f)&\le (2n-1)\pheight(f)+(2n-1)\log\left((m+1)(n+1)\sqrt{n}\right),
\end{align}
\end{lemma}

\begin{proof}
Estimate~(\ref{eprojres})  is due to Schmidt \cite[Lemma~4]{Sc90}. To prove~(\ref{eaffres}), we invoke Lemma~\ref{lsombradet}. 
Since $R_f(X)$ can be presented as a determinant of dimension ${2n-1}$, whose entries are polynomials of degree at most~$m$ and of affine height at most ${\aheight(f)+\log n}$, the result follows after an obvious calculation. \qed
\end{proof}

\begin{remark}
\label{rresx}
Estimate~(\ref{eaffres}) holds true also when ${m=0}$. We obtain the following statement: the resultant~$R_f$ of a polynomial $f(X)$ and its derivative $f'(X)$ satisfies
$$
\aheight(R_f)\le (2\deg f-1)\aheight(f) + (2\deg f-1)\log\left(2(\deg f)^2\right).
$$
\end{remark}

\subsection{Number fields and Discriminants}

We need some estimates for the discriminant of a number field in terms of the heights of its generators. 
 In this subsection $\K$ is  a number field, ${d=[\K:\Q]}$ and ${\norm(\cdot)=\norm_{\K/\Q}(\cdot)}$. 
The following result is due to Silverman \cite[Theorem 2]{Si84}.
\begin{lemma}
\label{lsilv}
Let ${\ua=(a_1,\ldots,a_k)}$ be a point in ${\bar\K^k}$.  Put ${\nu=[\K(\ua):\K]}$. Then
$$
\partial_{\K(\ua)/\K} \leq 2(\nu-1)\aheight(\ua)+\log \nu. \eqno\square
$$
\end{lemma}

This has the following consequence.

\begin{corollary}
\label{cmanydisc}
Let ${F(X)\in \K[X]}$ be a polynomial of degree~$N$. Then 
\begin{equation}
\label{emanydisc}
\sum_{F(\alpha)=0}\partial_{\K(\alpha)/\K} \le 2(N-1)\pheight(F) +3N\log N,
\end{equation}
the sum being over the roots of~$F$. 
\end{corollary}

\begin{proof}
Since for any root~$\alpha$ we have ${[\K(\alpha):\K]\le N}$, we estimate the left-hand side of~(\ref{emanydisc}) as 
$$
2(N-1)\sum_{F(\alpha)=0} \aheight(\alpha)+ N\log N 
$$
Remark~\ref{rheiroots} allows us to bound the sum on the right by ${\pheight(F)+\log(N+1)}$. Now, to complete the proof, just remark that ${(N-1)\log(N+1)\le N\log N}$. \qed
\end{proof}

\medskip

We shall also need a bound for the discriminant of a different nature, known as 
the \textsl{Dedekind-Hensel inequality} (see \cite[page~397]{De30} for historical comments and further references). This inequality gives an estimate of the relative discriminant of a number field extension in terms of the ramified places. 

\begin{lemma}
\label{ldehen}
Let~$\K$ be a number field of degree~$d$ over~$\Q$, and~$\L$ an extension of~$\K$ of finite degree~$\nu$, and let $\ram(\L/\K)$ be the set of places of~$\K$ ramified in~$\L$. 
Then 
\begin{equation}
\label{edehen}
\partial_{\L/\K} \le \frac{\nu-1}\nu \vheight\bigl(\ram(\L/\K)\bigr) + 1.26\nu.
\end{equation}
\end{lemma}

This is  Proposition~4.2.1 from~\cite{Bi97} (though the notation in~\cite{Bi97} is different, and the quantity estimated therein is $\nu\partial{\L/\K}$ in our notation), the only difference being that the error term  is now explicit. The proof is the same as in~\cite{Bi97}, but in the very last line one should use the estimate ${\sum_{p\le \nu}1\le 1.26\nu/\log\nu}$, which is~\eqref{epix}. 

A similar estimate was obtained by Serre \cite[Proposition~4]{Se81}. However,~\eqref{edehen} is more suitable for our purposes. 

It is useful to have an opposite estimate as well. The following lemma is obvious. 

\begin{lemma}
\label{lantidehen}
In the set-up of Lemma~\ref{ldehen} we have ${\vheight\bigl(\ram(\L/\K)\bigr)\le \nu \partial_{\L/\K}}$. 
\end{lemma}

This has the following consequence.

\begin{corollary}
\label{cramli}
Let ${\L_1, \ldots, \L_n}$ be a family of finite extensions of~$\K$  closed under the Galois conjugation over~$\K$. Then 
$$
\vheight\left(\bigcup_{i=1}^n \ram(\L_i/\K)\right) \le \sum_{i=1}^n\partial_{\L_i/\K}.
$$
\end{corollary}

\paragraph{Proof}
We may assume that the Galois action over~$K$ is transitive on ${\L_1, \ldots, \L_n}$ (otherwise, one obtains the estimate for every orbit of the Galois action and then sums the resulting inequalities up). In other words, the fields ${\L_1, \ldots, \L_n}$ form a full system of conjugates over~$K$, which means that 
$$
[\L_1:\K]=\ldots=[\L_n:\K]=n, \quad \ram(\L_1/\K)= \ldots=\ram(\L_n/\K), \quad  \partial_{\L_1/\K}= \ldots=\partial_{\L_n/\K}. 
$$
Hence 
$$
\vheight\left(\bigcup_{i=1}^n \ram(\L_i/\K)\right)= \vheight\bigl( \ram(\L_1/\K)\bigr)\le n  \partial_{\L_1/\K}= \sum_{i=1}^n\partial_{\L_i/\K}. 
\eqno\square
$$

\section{Power Series}
Our main technical tool is the quantitative Eisenstein theorem, based on the work of Dwork, Robba, Schmidt and van der Poorten \cite{dwork-robba,dwork-vdpoorten.eisenstein,Sc90}, in the form presented in~\cite{BB12}. 
Let 
\begin{equation}
\label{epuiseux}
y= \sum_{k=-k_{0}}^{\infty}a_{k} x^{k/e}
\end{equation}
be an algebraic power series with coefficients in~$\bar\Q$, where we assume $k_{0} \geq 0$ and $a_{-k_{0}} \ne 0$ when $k_{0} >0$. The classical Eisenstein theorem tells that the coefficients of this series belong to some number field, that for every valuation~$v$ of this field $|a_k|_v$ grows at most exponentially in~$k$, and for all but finitely many~$v$ we have ${|a_k|_v\le 1}$ for all~$k$. We need a quantitative form of this statement, in terms of an algebraic equation  ${f(x,y)=0}$ satisfied by~$y$. 

\subsection{Eisenstein Theorem}

Thus, let ${f(X,Y) \in \K(X,Y)}$ be a polynomial over a number field~$\K$. We put 
\begin{equation}
\label{edegsf}
d=[\K:\Q], \qquad m=\deg_Xf, \qquad n=\deg_Yf.
\end{equation}
Write
\begin{equation}
\label{ef01}
f(X,Y)=f_0(X)Y^n+f_1(X)Y^{n-1}+\ldots
\end{equation}
${[\L:\K] \leq n}$. 
Finally, for ${v\in M_\K}$ we denote by~$d_v$ its local degree over~$\Q$, and by~$\norm v$ its absolute norm:
\begin{equation}
\label{edv}
d_v=[\K_v:\Q_v], \qquad \norm v= \norm_{\K/\Q}(v).
\end{equation}
With this notation,  the height $\vheight(S)$ of a finite set of places ${S\subset M_\K}$ is given by ${d^{-1}\sum_{v\in S}d_v\log\norm v}$. 

The following is Theorem~6.3 from~\cite{BB12}. 

\begin{theorem}
\label{teisen}
Let~$\K$ be a number field and ${f(X,Y) \in \K(X,Y)}$ a separable polynomial. We use notation~\eqref{edegsf} and~\eqref{ef01}.  Let~$y$ be an algebraic power series, written as in~(\ref{epuiseux}), and satisfying ${f(x,y)=0}$. 
For every ${v\in M_\K}$ there exist real numbers ${A_{v}, B_{v} \geq 1}$, with ${A_{v} = B_{v} = 1}$
for all but finitely many~$v$, such that 
\begin{align}
\label{eav}
 d^{-1} \sum_{v\in M_{\K}}d_{v}\log A_{v} &\leq  3n\bigl(\pheight(f) + \log (mn) +3e\bigr), \\
 \label{ebv}
 d^{-1} \sum_{v \in M_{\K}}d_{v} \log B_{v} &\leq  \pheight(f)+2.
 \end{align}
and for any extension~$\barv$ of~$v$ to~$\bar\K$ we have
\begin{equation}
\label{eeis}
|a_{k}|_\barv \leq B_{v}A_{v}^{k/e-\lfloor-k_0/e\rfloor}   \qquad (k \geq -k_{0}).
 \end{equation}
\end{theorem}

\begin{remark}
We shall use this theorem only  in the ``integral case'' ${k_0=0}$, when~\eqref{eeis} becomes
 \begin{equation}
\label{eeisint}
|a_{k}|_{\barv} \leq B_{v}A_{v}^{k/e} \qquad (k \geq 0),
 \end{equation}
 but we prefer to state the theorem in full generality. 
\end{remark}

We will also use two consequences of this theorem, obtained in~\cite{BB12} as well. To state them, recall that 
the Puiseux theorem implies existence of ${n=\deg_Yf}$ distinct series ${y_1, \ldots, y_n}$, which can be written as
\begin{equation}
\label{ef_i}
y_i(x)=\sum_{k=-k_0(i)}^\infty a_{ik}x^{k/e_i} \qquad (i=1, \ldots, n),
\end{equation}
and which satisfy ${f\bigl(x,y_i(x)\bigl)=0}$. 

 We denote by ${D(X)=D_f(X)}$ the $Y$-discriminant of the polynomial $f(X,Y)$. Given a polynomial $P(X)$, we denote by $\ord_\alpha P(X)$  the order of~$\alpha$ as the root of $P(X)$.

The following proposition is composed from  Theorems~6.4 and~8.5  from~\cite{BB12}.

\begin{proposition}
\label{peisen}
Let
${f(X,Y)\in \K[X,Y]}$ be as above and let ${y_1, \ldots, y_n}$ be the~$n$ distinct series, written as in~\eqref{ef_i} and satisfying ${f\bigl(x,y_i(x)\bigl)=0}$.  

\begin{enumerate}
\item
\label{ibigv}
Let~$T$ be the (finite) set of ${v\in M_\K}$ such that ${|a_{ik}|_\barv>1}$ for some coefficient $a_{ik}$ and some extension~$\barv$ of~$v$ to~$\bar\K$. Then
\begin{equation}
\label{ehes}
\vheight(T) \le 3n\bigl(\pheight(f)+\log(mn)+1). 
\end{equation}

\item
\label{isumli}
The number fields ${\L_1, \ldots, \L_n}$, generated over~$\K$ by the coefficients of ${y_1, \ldots, y_n}$, respectively, satisfy
\begin{equation}
\label{esumordgen}
\sum_{i=1}^n\partial_{\L_i/\K} \le 8n \bigl(\ord_0D(X)+1\bigr)\bigl(\pheight(f) + 5n +\log m\bigr).
\end{equation}

\end{enumerate}
\end{proposition}

\subsection{The ``Essential'' Coefficients}
\label{ssess}
Let ${y\in \barq((x^{1/e}))}$ be a an algebraic power series written as in~\eqref{epuiseux}.  
We assume that~$e$ is smallest possible: ${y\notin \barq((x^{1/e'}))}$ for ${e'<e}$.

We define the $k$-th ramification index ${\epsilon_k=\epsilon_k(y)}$ as the smallest natural~$e'$ such that the  $k$-th partial sum
${y^{(k)}= \sum_{\ell=-k_{0}}^{k}a_\ell x^{\ell/e}}$ belongs to ${\barq((x^{1/e'}))}$. By the definition,
$$
\epsilon_{-k_0}=1, \qquad \epsilon_k\mid \epsilon_{k+1},
$$
and since~$e$ is smallest possible, we have ${\epsilon_k=e}$ for all sufficiently large~$k$. 

We call an index ${k>-k_0}$ \textsl{essential} if ${\epsilon_k>\epsilon_{k-1}}$ (that is, we ``gain new ramification'' with the term $a_kx^{k/e}$). The corresponding coefficient $a_k$ is called an \textsl{essential coefficient}. Clearly, an essential coefficient cannot be~$0$.

The series~$y$ can have only finitely many essential indices.  We want to estimate the sum of the heights of the essential coefficients. 
We denote by $\ord_0$ the discrete valuation on the local ring ${\barq[[x^{1/e}]]}$ normalized to have ${\ord_0(x)=1}$. 

{\sloppy

\begin{proposition}
\label{pess}
Let ${f(X,Y) \in \barq(X,Y)}$ be a separable polynomial. We use notation~\eqref{edegsf} and~\eqref{ef01}.  
Let~$y$ be an algebraic power series satisfying ${f(x,y)=0}$.
Assume that ${f_0(0) \ne 0}$. Then
\begin{equation}
\label{esumesss}
\sum_{\text{$k$ essential}}\vheight(a_k)\le \bigr(\pheight(f)+2\bigl)\log_2e+ 3n\bigl(\pheight(f) + \log (mn) +3e\bigr)\ord_0\bigl(f'_Y(x,y)\bigr)
\end{equation}
\end{proposition}

}

If ${f_0(0) \ne 0}$ then the the series $y(x)$ is  integral over the ring $\barq[[x]]$ and can be written as 
\begin{equation}
\label{epuiseuxint}
y(x)=\sum_{k=0}^\infty a_{k}x^{k/e} 
\end{equation}  
The assumption ${f_0(0)\ne 0}$ is purely technical; a similar result  holds in general as well. However, without this assumption estimate~\eqref{esumesss} gets weaker than we need, while assuming ${f_0(0)\ne 0}$ does not hurt generality: see Section~\ref{sproofcw}. 

The proof of Proposition~\ref{pess} relies on the following lemma (which is an analog of Lemma~7.2 in~\cite{BB12}). 

\begin{lemma}
\label{less}
Assume that ${f_0(0) \ne 0}$. Then there is at most $\log_2e$  essential indices, and their sum does not exceed $e\,\ord_0\bigl(f'_Y(x,y)\bigr)$. 
\end{lemma}
\paragraph{Proof}
Since ${\epsilon_{k-1}\mid \epsilon_k}$, we have ${\epsilon_k\ge 2\epsilon_{k-1}}$ whenever~$k$ is essential, which means that there can be at most $\log_2e$ essential indices.

Now let us prove the statement about the sum. 
Together with the series~$y$ we consider the ``twisted series'' 
$$
\sum_{k=0}^\infty a_k\zeta^{(j-1)k}x^{k/e} \in \barq[[x^{1/e}]] \qquad (j=1, \ldots , e), 
$$
where~$\zeta$ is a primitive $e$-th root of unity. These~$e$ series are among the~$n$ distinct series ${y_1, \ldots, y_n}$, which satisfy ${f(x,y_i)=0}$,  and after  re-numbering  we may assume that 
$$
y_j= \sum_{k=0}^\infty a_k\zeta^{(j-1)k}x^{k/e} \in \barq[[x^{1/e}]] \qquad (j=1, \ldots , e).
$$
In particular, ${y=y_1}$. 

By the definition of~$\epsilon_k$ we have  ${y_j^{(k)}=y_{j'}^{(k)}}$ if and only if ${j\equiv j'\mod \epsilon_k}$. In particular, ${y_j^{(k)}=y^{(k)}}$ if and only if ${\epsilon_k\mid (j-1)}$. We partition the set ${J=\{2,3,\ldots, e\}}$ as
$$
J=J_1\cup J_2\cup J_3\ldots , \qquad J_k\cap J_\ell =\varnothing \quad (k\ne \ell)
$$
where 
$$
J_k= \bigl\{j\in J: \epsilon_{k-1}\mid (j-1), \ \epsilon_k\nmid(j-1)\bigl\}. 
$$
The following two observations are now crucial:
\begin{itemize}
\item
for ${j\in J}$ we have ${\ord_0(y-y_j)=k/e}$ if and only if ${j\in J_k}$;

\item
the set~$J_k$ is not empty if and only if~$k$ is an essential index for~$y$. 

\end{itemize}
Using this, we find 
\begin{equation}
\label{esumess}
\sum_{\text{$k$ essential}}\frac ke \le \sum_{k=0}^\infty\frac ke|J_k| =\ord_0\left(\prod_{j=2}^e (y-y_j)\right)
\end{equation}
Since ${f_n(0)\ne0}$, all the series ${y_1, \ldots, y_n}$ are integral over $\barq[[x]]$. Hence the product in the right-hand side of~\eqref{esumess} divides
${f'_Y(x,y)=f_0(x)\prod_{j=2}^n (y-y_j)}$. 
It follows that the right-hand side of~\eqref{esumess} does not exceed $\ord_0\bigl(f'_Y(x,y)\bigr)$, which proves the lemma. \qed

\paragraph{Proof of Proposition~\ref{pess}}
By Theorem~\ref{teisen} we have
$$
\vheight(a_k) \le  \pheight(f)+2+ \frac ke \cdot 3n\bigl(\pheight(f) + \log (mn) +3e\bigr). 
$$
Hence
$$
\sum_{\text{$k$ essential}}\vheight(a_k)\le \bigr(\pheight(f)+2\bigl)\sum_{\text{$k$ essential}} 1+ 3n\bigl(\pheight(f) + \log (mn) +3e\bigr)\sum_{\text{$k$ essential}}\frac ke.
$$
We conclude, applying the lemma. \qed

\medskip

Now assume that~$\K$ is a number field and~$y$ a series with coefficients in~$\K$.  We denote by $\Ess(y)$ the set of places ${v\in M_\K}$ such that ${|a_k|_v<1}$ for some essential coefficient~$a_k$ of~$y$:
$$
\Ess(y) =\{v\in M_K: \text{ there exists an essential index~$k$ such that $|a_k|_v<1$} \}
$$

{\sloppy

\begin{proposition}
\label{pessen}
Let
${f(X,Y)\in \K[X,Y]}$ be a separable polynomial. We use notation~\eqref{edegsf} and~\eqref{ef01}.  
Assume that ${f_0(0) \ne 0}$. Let ${y_1, \ldots, y_n}$ be the~$n$ distinct series,  satisfying ${f\bigl(x,y_i(x)\bigl)=0}$.  Assume that the coefficients of all these series belong to~$\K$. 
Then
\begin{equation} 
\label{ehess}
\vheight\left(\bigcup_{i=1}^n\Ess(y_i)\right)\le n\bigl( \pheight(f)+2\bigl)+ 3n\bigl(\pheight(f) + 4n +\log m \bigr)\ord_0D(X).
\end{equation}
where $D(X)$ is the $Y$-discriminant of $f(X,Y)$. 
\end{proposition}

}

\paragraph{Proof}
Recall that  the series~$y_i$ has $e_i$ ``twists'' among ${y_1, \ldots, y_n}$, as defined in the proof of Lemma~\ref{less}. If~$y_j$ is a twist of~$y_i$ then each coefficient of~$y_j$ is equal to the corresponding coefficient of~$y_i$ times an $e_i$-th root of unity, which implies that ${\Ess(y_i)=\Ess(y_j)}$. 

Select a maximal subset from ${\{y_1, \ldots, y_n\}}$ such that none of its elements is a twist of the other. After re-numbering, we may assume that this subset is ${\{y_1, \ldots, y_s\}}$ (this is \textsc{not} the numbering adopted in the  proof of Lemma~\ref{less}). Then each of ${y_1, \ldots, y_n}$ is a twist of one of ${y_1, \ldots, y_s}$, which implies that 
${\bigcup_{i=1}^n\Ess(y_i)=\bigcup_{i=1}^s\Ess(y_i)}$
and ${e_1+\ldots+e_s=n}$.

Item~\ref{idenum} of Proposition~\ref{pvhei} implies that ${\vheight\bigl(\Ess(y_i)\bigr)}$ is bounded by the sum of the heights of the essential coefficients of~$y_i$. Now, using Proposition~\ref{pess} we obtain
\begin{align*}
\vheight\left(\bigcup_{i=1}^n\Ess(y_i)\right)&=\vheight\left(\bigcup_{i=1}^s\Ess(y_i)\right)\\
& \le \sum_{i=1}^s\Bigl( \bigr(\pheight(f)+2\bigl)\log_2e_i+ 3n\bigl(\pheight(f) + \log (mn) +3e_i\bigr)\ord_0\bigl(f'_Y(x,y_i)\bigr)\Bigr)\\
& \le \bigl( \pheight(f)+2\bigl)\sum_{i=1}^s\log_2e_i+ 3n\bigl(\pheight(f) + 4n +\log m \bigr)\sum_{i=1}^s\ord_0\bigl(f'_Y(x,y_i)\bigr)\\
&\le \bigl( \pheight(f)+2\bigl)\sum_{i=1}^se_i+ 3n\bigl(\pheight(f) + 4n +\log m \bigr)\sum_{i=1}^n\ord_0\bigl(f'_Y(x,y_i)\bigr)\\
&= n\bigl( \pheight(f)+2\bigl)+ 3n\bigl(\pheight(f) + 4n +\log m \bigr)\ord_0D(X),
\end{align*}
as wanted. \qed


\section{Proximity and Ramification}
\label{sprox}

This section is the technical heart of the article. We consider a covering ${\CC\stackrel x\to \PP^1}$, defined over a number field~$\K$, and call a point ${P\in \CC(\bar\K)}$ \textsl{semi-defined} over~$\K$ if ${x(P)\in \PP^1(\K)}$. We define a finite set~$\QQ$ of  points from  $\CC(\bar\K)$ (which include the finite ramified points of the covering~$x$, but may contain some other points as well) and prove two statements (Propositions~\ref{pclose} and~\ref{pram} below) which, informally, assert the following.

\begin{itemize}

\item
If a finite place ${v\in M_\K}$ ramifies in the field $\K(P)$ (where~$P$ is semi-defined over~$\K$) then (unless~$v$ is ``bad'' in certain sense) the point~$P$ must be ``$v$-adically close'' to a point from the set~$\QQ$ (Propositions~\ref{pclose}). 

\item
Given a point~$Q$ on~$\CC$ and a finite place~$v$ (again, it should not be ``bad'' in some sense), for the points~$P$ (semi-defined over~$\K$)  in a ``$v$-adic neighborhood'' of~$Q$, the $v$-ramification in the field $\K(P)$ is determined by the ``$v$-adic distance'' between~$P$ and~$Q$ and the ramification of the point~$Q$ over~$\PP^1$. Roughly speaking, ``geometric ramification determines arithmetic ramification'' (Propositions~\ref{pram}). 

\end{itemize}

It is not difficult to make qualitative statements of this kind, but it is a rather delicate task to make everything explicit. In particular, we will explicitly estimate (Proposition~\ref{pquant}) the set of the ``bad'' places.

\subsection{Proximity}
\label{ssprox}

Now let us be precise. In this section we fix, once and for all:

\begin{itemize}
\item
a number field~$\K$;
\item
an absolutely irreducible smooth projective curve~$\CC$ defined over~$\K$;

\item
a non-constant rational function ${x\in \K(\CC)}$;

\item
one more rational function ${y\in \K(\CC)}$ such that ${\K(\CC)=\K(x,y)}$ (existence of such~$y$ follows from the primitive element theorem). 
\end{itemize}

Let ${f(X,Y)\in \K[X,Y]}$ be the $\K$-irreducible polynomial such that ${f(x,y)=0}$ (it is well-defined up to a constant factor). Since~$\CC$ is absolutely irreducible, so is the polynomial $f(X,Y)$.

We put ${m=\deg_Xf}$, ${n=\deg_Yf}$, and write
\begin{equation}
\label{efxy}
f(X,Y)= f_0(X)Y^n+f_1(X)Y^{n-1}+ \cdots+ f_n(X).
\end{equation}

Let ${Q\in \CC(\bark)}$ be a finite $\bark$-point of~$\CC$ (``finite'' means that~$Q$ is not a pole of~$x$). We set ${\alpha=x(Q)}$ and we denote by~$e_Q$ the ramification index of~$x$ at~$Q$ (that is, ${e_Q=\ord_Q(x-\alpha)}$). When it does not cause a confusion  we  write~$e$ instead of~$e_Q$.  Fix a primitive $e$-th root of unity ${\zeta=\zeta_e}$. Then there exist~$e$ equivalent Puiseux expansions of~$y$ at~$Q$:
\begin{equation}
\label{epuis}
y^{(Q)}_j= \sum_{k = -k^{(Q)}}^{\infty} a_k^{(Q)}\zeta^{(j-1)k}(x-\alpha)^{k/e} \qquad (j=1, \ldots, e),
\end{equation}
where ${k^{(Q)}=\max\left\{0, -\ord_Q(y)\right\}}$.  

Let~$\barv$ be a place of~$\bark$. We say that the series~(\ref{epuis}) converge  $\barv$-adically  at ${\xi\in \bark}$, if, for a fixed $e$-th root ${\sqrt[e]{\xi-\alpha}}$, the~$e$ numerical series 
$$
\sum_{k = -k^{(Q)}}^{\infty} a_k^{(Q)}\left(\zeta^{j-1}\sqrt[e]{\xi-\alpha}\right)^k \qquad (j=1, \ldots, e)
$$
converge in the $\barv$-adic topology. We denote by $y^{(Q)}_j(\xi)$, with ${j=1, \ldots, e}$, the corresponding sums. While the individual sums depend  on the particular choice of the root 
${\sqrt[e]{\xi-\alpha}}$,  the very fact of convergence, as well as the set ${\left\{y^{(Q)}_1(\xi),\ldots, y^{(Q)}_{e}(\xi)\right\}}$ of the sums, are independent of the choice of the root.

Now we are ready to introduce the principal notion of this section, that is of \textsl{proximity of a point to a different point} with respect to a given place ${\barv\in M_\bark}$. 

\begin{definition}
\label{dprox}
Let ${P\in \CC(\bark)}$ be a finite $\bark$-point of~$\CC$, and put ${\xi=x(P)}$. We say that~$P$ is \textbf{$\barv$-adically close} to~$Q$ if the following conditions are satisfied:

\begin{itemize}
\item
${|\xi-\alpha|_{\barv}<1}$;

\item
the~$e$ series~(\ref{epuis}) $\barv$-adically converge at~$\xi$, and one of the sums $y^{(Q)}_j(\xi)$ is equal to $y(P)$.

\end{itemize}
\end{definition}

\textbf{An important warning:} the notion of proximity just introduced is not symmetric in~$P$ and~$Q$: the proximity of~$P$ to~$Q$ does not imply, in general, the proximity of~$Q$ to~$P$. Intuitively, one should think of~$Q$ as a ``constant'' point, and of~$P$ as a ``variable'' point.

To state the main results of this section, we have to define a finite set~$\QQ$ of $\bark$-points of the curve~$\CC$, and certain  finite sets of ``bad'' places of the field~$\K$. Let~ ${R(X)=R_f(X)\in \K[X]}$   be the $Y$-resultant of $f(X,Y)$ and $f'_Y(X,Y)$, and let~$\AA$ be the set of the roots of $R(X)$:
$$
\AA=\{\alpha \in \bark : R(\alpha)=0\}.
$$
We define~$\QQ$ as follows:
$$
\QQ=\left\{Q\in \CC(\bark) : x(Q)\in \AA\right\}.
$$
It is important to notice that~$\QQ$ contains all the finite ramification points of~$x$ (and may contain some other points as well).  Also, the set~$\QQ$ is Galois-invariant over~$\K$: every point belongs to it together with its Galois orbit over~$\K$.

Now let us define the finite sets of ``bad'' places of~$\K$ mentioned above. 
First of all  we assume (as we may, without loss of generality) that 
\begin{equation}
\label{ef0monic}
\text{the polynomial $f_0(X)$, defined in~(\ref{efxy}), is monic.}
\end{equation} 
In particular,~$f$ has a coefficient equal to~$1$, which  implies equality of the affine and the projective heights of~$f$:
\begin{equation}
\label{epah}
\aheight(f)=\pheight(f).
\end{equation}
Now, we define
\begin{align*}
T_1&=\left\{v\in M_\K^0 : \text{the prime below $v$ is $\le n$}\right\},\\
T_2&=\left\{v\in M_\K^0 : |f|_v>1\right\}.
\end{align*}
Further, let~$r_0$ be the leading coefficient of $R(X)$. We define
$$
T_3 =\left\{v\in M_\K^0 : |r_0|_v<1\right\}.
$$
Next, we let~$\Delta$ be the resultant of $\hatr(X)$ and $\hatr'(X)$, where  $\hatr$ is the radical of~$R$, see Subsection~\ref{ssloc}. Since the polynomial $\hatr(X)$ is separable, we have ${\Delta \in \K^\ast}$. Now we define the set~$T_4$ as follows:
$$
T_4= \left\{v\in M_\K^0 : |\Delta|_v<1\right\}.
$$

The sets~$T_5$ and~$T_6$ will be defined under the assumptions
\begin{gather}
\label{eallin}
\QQ\subset \CC(\K), \\
\label{erootsin}
\text{$\K$ contains $e_Q$-th roots of unity for all ${Q\in \QQ}$}. 
\end{gather}
Notice that~\eqref{eallin} implies that 
\begin{equation}
\label{eaain}
\AA\subset \K.
\end{equation}

Now fix ${Q\in \CC(\K)}$ and define the sets~$T_5^{(Q)}$ and~$T_6^{(Q)}$  using the Puiseux expansions of~$y$ at ${Q\in\QQ}$. As in~(\ref{epuis}), we denote by~$a_k^{(Q)}$ the coefficients of these expansions; by~\eqref{erootsin} we may assume that these coefficients are in~$\K$. 
Now define
\begin{align*}
T_5^{(Q)}=\left\{v\in M_\K^0 : \text{$\bigl|a_k^{(Q)}\bigr|_v>1$ for   some~$k$}\right\},\qquad
T_5=\bigcup_{Q\in \QQ}T_5^{(Q)}.
\end{align*}
The Eisenstein theorem implies that the set~$T_5^{(Q)}$ is finite.

Finally, put
$$
T_6^{(Q)}=\Ess(y_1^{(Q)}), \qquad
T_6=\bigcup_{Q\in \QQ}T_6^{(Q)}.
$$
where $\Ess(y)$ is defined in Subsection~\ref{ssess} (just before Proposition~\ref{pessen}) as the set of places ${v\in M_\K}$ such that ${|a_k|_\barv<1}$ for some essential coefficient~$a_k$.

Recall that a point ${P\in \CC(\bark)}$ is \textsl{semi-defined over~$\K$} if ${\xi=x(P)\in \PP^1(\K)}$.
Let ${P,Q\in \CC(\bark)}$ be semi-defined over~$\K$, let  ${v\in M_\K}$ be a finite place of~$\K$ 
and~$\pi$ a primitive element of the local ring~$\OO_v$. 
Define
\begin{equation}
\label{eell}
\ell(P,Q,v)=\frac{\log |\xi-\alpha|_v}{\log |\pi|_v}=\ord_\pi(\xi-\alpha),
\end{equation}
where, as above, ${\xi=x(P)}$ and ${\alpha=x(Q)}$.  (Since both~$P$ and~$\Q$ are semi-defined over~$\K$, we have ${\xi,\alpha \in \K}$.)

Now we are ready to state the principal results of this section. 
 We call a point~$P$ \textsl{finite} if it  is not a pole of~$x$.

\begin{proposition}
\label{pclose}
Let~$\QQ$ be the  set defined above, and assume~\eqref{eallin},~\eqref{erootsin}. Let ${P\in \CC(\bark)\smallsetminus\QQ}$ be  a finite point  semi-defined over~$\K$, and  ${v\in M_\K}$ a finite place of~$K$.   Assume that $|\xi|_{v} \le 1$ (we again put ${\xi=x(P)}$) and that 
${v \notin T_2\cup T_3\cup T_4 \cup T_5}$. Let ${\barv \in M_\bark}$, extending~$v$, be such that its restriction to the field $\K(P)$ is ramified over~$\K$. Then  our point~$P$   is $\barv$-adically close to some (unique) ${Q \in \QQ}$. 
\end{proposition}

\begin{proposition}
\label{pram}
Let ${P\in \CC(\bark)}$ be a finite point semi-defined over~$\K$, and assume that~$P$
is  $\barv$-adically close to some ${Q\in \CC(\K)}$ for some finite place ${\barv\in M_\bark}$.  Let~$v$ and~$w$ be the restrictions of~$\barv$ to~$\K$ and $\K(P)$, respectively. Assume that~$v$ does not belong to ${T_1\cup T_5^{(Q)}\cup T_6^{(Q)}}$, and that~$\K$ contains the $e_Q$-th roots of unity. Then the ramification index of~$w$ over~$v$ is equal to $e_Q/(\gcd(e_Q,\ell)$, where 
${\ell=\ell(P,Q,v)}$ is defined in~(\ref{eell}). 
\end{proposition}
Intuitively, the last condition means that the ``arithmetic ramification is determined by the geometric ramification''.

\begin{proposition}
\label{pquant}
Assume~\eqref{eallin} and~\eqref{erootsin}.
Put
${T=T_1\cup T_2\cup\ldots\cup T_6}$.
Assume further that 
\begin{equation}
\label{eunramf0}
\text{the covering ${\CC\stackrel x\to \PP^1}$ does not ramify over the roots of $f_0(X)$}.
\end{equation} 
Then
$$
\vheight(T) \le  52mn^2 \bigl(\pheight(f) + 2m+2n\bigr).
$$
\end{proposition}

Finally, if we do \textsl{not} assume~\eqref{eallin} and~\eqref{erootsin}, then we have to estimate the smallest extension of~$K$ satisfying~\eqref{eallin} and~\eqref{erootsin}.

\begin{proposition}
\label{pcompositum}
Let~$\L$ be the compositum of the fields $\K(Q)$ and the fields generated over~$\K$ by $e_Q$-th  roots of unity, for all ${Q\in\QQ}$.
 Then
\begin{equation}
\label{eparlk}
\partial_{\L/\K}\le 105 mn^2\bigl(\pheight(f)+2m+2n\bigr). 
\end{equation}
\end{proposition}

\subsection{Proof of Proposition~\ref{pclose}}
We fix, once and for all, a finite place ${v\in M_\K}$, its extension ${\barv\in M_\bark}$, and a point ${P\in \CC(\bark)}$ semi-defined over~$\K$ and such that ${\xi=x(P)\notin\AA}$. We assume that ${|\xi|_v\le 1}$, that ${v\notin T_2\cup\ldots\cup T_5}$ and that the restriction~$w$ of $\barv$
 to $\K(P)$ is ramified over~$\K$. We shall prove that~$P$ is $\barv$-adically close to a unique ${Q\in \QQ}$, which depends on~$w$, but not on~$\barv$.

Since ${v\notin T_2\cup T_3}$, the polynomial $R(X)$ belongs to $\OO_v[X]$ and is $v$-monic. Lemma~\ref{lfcap} implies that so is its radical $\hatr(X)$. Also, every root~$\alpha$ of~$R$ is a $v$-adic integer.   

Put ${\eta =y(P)}$. Since ${\xi\notin \AA}$, the point $(\xi, \eta)$ of the plane curve ${f(X,Y)=0}$ is non-singular, which implies that ${\K(P)=\K(\xi, \eta)=\K(\eta)}$ (recall that ${\xi\in \K}$). Now Lemma~\ref{lram} implies that ${\left|f_Y'(\xi,\eta)\right|_\barv<1}$. It follows that ${|R(\xi)|_v<1}$, which implies that ${|\hatr(\xi)|_v<1}$ by Lemma~\ref{lfcap}.

Next, since ${v\notin T_4}$, we have ${|\hatr'(\xi)|_v=1}$. Lemma~\ref{lhens} implies now that there exists a unique ${\alpha \in \AA}$ such that ${|\xi-\alpha|_\barv<1}$. 

Fix this~$\alpha$ from now on. There is
${\sum_{x(Q)=\alpha}e_Q=n}$  Puiseux expansions of~$y$ at the points~$Q$ above~$\alpha$, and they satisfy
$$
f(x,Y)= f_0(x)\prod_{x(Q)=\alpha}\prod_{j=1}^{e_Q}\left(Y-y^{(Q)}_j\right).
$$
Since ${v\notin T_5}$, each of the series $y^{(Q)}_j$ has $v$-adic convergence radius at least~$1$. Since ${|\xi-\alpha|_\barv<1}$, all them $\barv$-adically  converge at~$\xi$. Moreover, the convergence is absolute, because~$\barv$ is non-archimedean. Hence 
$$
f(\xi,Y)= f_0(\xi)\prod_{x(Q)=\alpha}\prod_{j=1}^{e_Q}\left(Y-y^{(Q)}_j(\xi)\right).
$$
Since ${R(\xi)\ne 0}$, we have ${f_0(\xi)\ne 0}$ as well. Hence we have on the left and on the right polynomials of degree~$n$ in~$Y$, the polynomial on the left having ${\eta=y(P)}$ as a simple root (here we again use that ${R(\xi)\ne 0}$). Hence exactly one of the sums $y^{(Q)}_j(\xi)$ is equal to~$\eta$. We have proved that~$P$ is $\barv$-adically close to exactly one ${Q\in \QQ}$. \qed

\subsection{Proof of Proposition~\ref{pram}}

We may assume, by re-defining the root $\sqrt[e]{\xi-\alpha}$ that ${\eta=y(P)}$ is the sum of ${y_1^{(Q)}}$ at~$\xi$. In the sequel we omit reference to~$Q$ (when it does not lead to confusion) and write~$e$ for~$e_Q$, $a_k$ for $a_k^{(Q)}$, etc. Thus, we have, in the sense of $\barv$-adic convergence,
\begin{equation}
\label{eetav}
\eta=\sum_{k= -k^{(Q)}}^\infty a_k\left(\sqrt[e]{\xi-\alpha}\right)^k.
\end{equation}
Let~$v$ and~$w$ be the restrictions of~$\barv$ to~$\K$ and $\K(P)$, respectively. We assume that~$v$ does not belong to ${T_1\cup T_5^{(Q)}\cup T_6^{(Q)}}$.  Put
$$
e'= \frac e{\gcd(e,\ell)}, \qquad \ell'= \frac \ell{\gcd(e,\ell)},
$$
where ${e=e_Q}$ and ${\ell=\ell(P,Q,v)}$ is defined in~(\ref{eell}). We have to show that the ramification index of~$w$ over~$v$ is equal to~$e'$. 

Recall that by the assumption  ${Q\in \CC(\K)}$ and~$\K$ contains $e$-th roots of unity. It follows that ${\alpha =x(Q)\in \K}$ and that~$\K$ contains the coefficients  of the Puiseux expansions of~$y$ at~$Q$. 

Let~$\K_v$ be a $v$-adic completion of~$\K$. We consider~$\bark_\barv$ as its algebraic closure, and the fields ${\K_v(P)=\K_v(\eta)}$ and ${\K_v\left(\sqrt[e]{\xi-\alpha}\right)}$ as subfields of the latter. According to~(\ref{eetav}), we have ${\K_v(\eta)\subset \K_v\left(\sqrt[e]{\xi-\alpha}\right)}$. The latter field has ramification~$e'$ over~$\K_v$ by item~\ref{irad} of Lemma~\ref{lramm}. (The assumption ${v\notin T_1}$ implies that~$e$ is not divisible by the characteristic of the residue field.)

Assume that the ramification of $\K_v(\eta)/\K_v$ is not~$e'$. Then there exists a prime divisor~$q$ of~$e'$ such that the ramification index of $\K_v(\eta)/\K_v$ divides $e'/q$. We want to show that this is impossible.

Let~$\kappa$ be the smallest~$k$ with the properties ${a_k\ne 0}$ and ${q\nmid \kappa}$.  Then~$\kappa$ is an \textsl{essential index} of the series~$y_1$ as defined in Subsection~\ref{ssess}, and~$a_\kappa$ is an essential coefficient. Put
$$
\theta = \eta-\sum_{k=k^{(Q)}}^{\kappa-1} a_k\left(\sqrt[e]{\xi-\alpha}\right)^k = a_\kappa\left(\sqrt[e]{\xi-\alpha}\right)^\kappa+ \sum_{k=\kappa+1}^\infty a_k\left(\sqrt[e]{\xi-\alpha}\right)^k.
$$
By the definition of~$\kappa$, we have 
${\theta \in \K_v\left(\eta, \sqrt[e/q]{\xi-\alpha}\right)}$. The ramification of $\K_v\left( \sqrt[e/q]{\xi-\alpha}\right)/\K_v$ is ${(e/q)/\gcd(e/q, \ell)}$ (we again use item~\ref{irad} of Lemma~\ref{lramm}). 
Since~$q$ divides $e'$, it cannot divide~$\ell'$, and we have ${\gcd(e/q, \ell)=\gcd(e,\ell)}$, which implies that ${(e/q)/\gcd(e/q, \ell)=e'/q}$.

Thus, the ramification of $\K_v\left( \sqrt[e/q]{\xi-\alpha}\right)/\K_v$ is $e'/q$, and the ramification  of $\K_v(\eta)/\K_v$ divides $e'/q$. Item~\ref{icompos} of Lemma~\ref{lramm} now implies that the ramification of $\K_v\left(\eta, \sqrt[e/q]{\xi-\alpha}\right)/\K_v$ is $e'/q$. Hence the ramification of $\K_v(\theta)/\K_v$ divides $e'/q$, which implies that ${\ord_\pi\theta\in (q/e')\Z}$.

But, since ${v\notin T_5^{(Q)}\cup T_6^{(Q)}}$, we have ${|a_k|_v\le 1}$ for all~$k$ and ${|a_\kappa|_v=1}$, which implies that 
${|\theta|_v=\left|\left(\sqrt[e]{\xi-\alpha}\right)^\kappa\right|_v}$. It follows that 
$$
\ord_\pi\theta= \frac\kappa e\ord_\pi(\xi-\alpha)=  \frac{\kappa\ell} e= \frac {\kappa\ell'}{e'} . 
$$
We have proved that ${\kappa\ell/e' \in (q/e')\Z}$. 
But~$q$ does not divide any of the numbers~$\kappa$ and~$\ell'$, a contradiction. \qed

\subsection{Proof of Proposition~\ref{pquant}}
The proposition is a direct consequence of the estimates
\begin{align}
\label{et1}
\vheight(T_1)& \le 1.02n,\\
\label{et2}
\vheight(T_2) & \le \pheight(f),\\
\label{et3}
\vheight(T_3) & \le (2n-1)\bigl(\pheight(f) + m\log2+ \log(2n^2)\bigr),\\
\label{et4}
\vheight(T_4) & \le 16mn^2 \bigl(\pheight(f) + 2m +2\log n\bigr),\\
\label{et5}
\vheight(T_5) & \le 14mn^2\bigl(\pheight(f)+2m+2n\bigr),\\
\label{et6m}
\vheight\bigl(T_6) & \le 18mn^2\bigl(\pheight(f)+2m+\log n\bigr).
\end{align}

\begin{remark}
Assumption~\eqref{eunramf0} is used only in the proof of~\eqref{et6m}.
\end{remark}

\paragraph{Proof of~(\ref{et1})}
Obviously,
${\vheight(T_1) \le \sum_{p\le n}\log p}$, which is bounded by
$1.02n$ according to~\eqref{ethetax}. \qed

\paragraph{Proof of~(\ref{et2})}
Item~\ref{idenum} of Proposition~\ref{pvhei} implies that ${\vheight(T_2)\le \aheight(f)}$.  Since ${\aheight(f)=\pheight(f)}$ by~\eqref{epah}, the result follows.\qed

\paragraph{Proof of~(\ref{et3})}
Item~\ref{idenum} of Proposition~\ref{pvhei} and Lemma~\ref{lheires} imply that
\begin{equation}
\label{er0}
\vheight(T_3)\le\aheight(r_0)\le \aheight(R) \le (2n-1)\aheight(f) + (2n-1)\left(\log(2n^2)+ m\log2\right).
\end{equation}
Again using~(\ref{epah}),  we have the result.  \qed

\paragraph{Proof of~(\ref{et4})}
We have
${\deg \hatr\le \deg R\le (2n-1)m}$. Further, using Corollary~\ref{cdivi} and inequalities~(\ref{er0}), we find
$$
\aheight(\hatr) \le \pheight(R)+\aheight(r_0)+\deg R \leq (4n-2)\aheight(f)+(8n-4)\left(\log n+m\right).
$$
Finally, using Remark~\ref{rresx} and the previous estimates, we obtain
$$
\vheight(T_4)\le \aheight(\Delta) \le (2\deg\hatr-1)\left(\aheight(\hatr)+ \log(2(\deg\hatr)^2)\right)
\le 16mn^2 \aheight(f) + 32mn^2\left(\log n+m\right).
$$
Using~(\ref{epah}),  we obtain the result.  \qed

\paragraph{Preparation for the proofs of~(\ref{et5}) and~(\ref{et6m})}
Recall that we denote by $R(X)$ the $Y$-resultant of $f(X,Y)$ and $f'_Y(X,Y)$ and by~$\AA$ the set of the roots of ${R(X)}$. Then
\begin{gather}
\label{ecardaa}
|\AA|\le\sum_{\alpha \in \AA} \ord_\alpha R(X)\le \deg R(X)\le m(2n-1),\\
\sum_{\alpha\in \AA}\aheight(\alpha) \le \sum_{\alpha\in \AA}\ord_\alpha R(X)\aheight(\alpha)
\le \pheight(R)+\log(2mn)
\label{esumheal}
\le (2n-1)\pheight(f)+3n\log (4mn),
\end{gather}
where for~(\ref{esumheal}) we use Remark~\ref{rheiroots} and Lemma~\ref{lheires}. 
Using  the notation
${f^{(\alpha)}(X,Y)=f(X+\alpha,Y)}$
and Corollary~\ref{cfalpha}, we obtain the  inequality
\begin{equation}
\label{esumfal}
\sum_{\alpha\in \AA}\aheight(f^{(\alpha)})\le\sum_{\alpha\in \AA}\ord_\alpha R(X)\aheight(f^{(\alpha)}) \le 4mn\pheight(f)+7 m^2 n + 3nm\log n.
\end{equation}

\paragraph{Proof of~(\ref{et5})}  
Fix ${\alpha \in \K}$. The height of the set 
${T_5^{(\alpha)}=\bigcup_{x(Q)=\alpha}T_5^{(Q)}}$
can be estimated using  item~\ref{ibigv} of Proposition~\ref{peisen} with polynomial $f^{(\alpha)}$ instead of~$f$. We obtain
\begin{equation}
\label{et5alpha}
\vheight\bigl( T_5^{(\alpha)}\bigr)\le  3n\bigl(\pheight(f^{(\alpha)})+\log(mn)+1).
\end{equation}
The set~$T_5$ is contained in the union of all $T_5^{(\alpha)}$ with ${\alpha \in \AA}$. Hence combining~\eqref{ecardaa},~\eqref{esumfal} and~\eqref{et5alpha}, we obtain
\begin{align*}
\vheight(T_5) &\le 3n\left(\sum_{\alpha\in \AA}\pheight(f^{(\alpha)})+(\log(mn)+1)|\AA|\right)
\le14mn^2\bigl(\pheight(f)+2m+2n\bigr),
\end{align*}
as wanted.\qed

\paragraph{Proof of~(\ref{et6m})}
It is totally analogous to the proof of~\eqref{et5}.  We define 
${T_6^{(\alpha)}= \bigcup_{x(Q)=\alpha}\Ess(y^{(Q)})}$. If the set $T_6^{(\alpha)}$ is non-empty then  the covering ${\CC\stackrel x\to \PP^1}$ ramifies over~$\alpha$, and condition~\eqref{eunramf0} implies that ${f_0^{(\alpha)}(0)=f_0(\alpha)\ne 0}$. Hence we may apply~\eqref{ehess} with~$f^{(\alpha)}$ instead of~$f$. We obtain
$$
\vheight\bigl(T_6^{(\alpha)}\bigr)\le n\bigl( \pheight(f^{(\alpha)})+2\bigl)+ 3n\bigl(\pheight(f^{(\alpha)}) + 4n +\log m \bigr)\ord_\alpha D(X)
$$ 
Next, we use~\eqref{ecardaa} and~\eqref{esumfal} to obtain
\begin{align*}
\vheight(T_6) &\le n\bigl( \sum_{\alpha\in \AA}\pheight(f^{(\alpha)})+2|\AA|\bigl)
+ 3n\left(\sum_{\alpha\in \AA}\ord_\alpha R(X)\pheight(f^{(\alpha)}) + (4n +\log m)\sum_{\alpha\in \AA}\ord_\alpha R(X) \right)\\
&\le 18mn^2\bigl(\pheight(f)+2m+\log n\bigr),
\end{align*}
as wanted. This completes the proof of Proposition~\ref{pquant}. \qed

\subsection{Proof of Proposition~\ref{pcompositum}}
We have 
\begin{equation}
\partial_{\L/\K}\le \sum_{\alpha\in \AA}\partial_{\K(\alpha)/\K}+ \sum_{\alpha\in \AA} \sum_{x(Q)=\alpha}\partial_{\K(\alpha)(Q)/\K(\alpha)}+\sum_{r=1}^n\partial_{\K(\zeta_r)/\K},
\end{equation}
where~$\zeta_r$ is a primitive $r$-th root of unity.

Each ${\alpha \in \AA}$ generates over~$\K$ a field of degree at most ${\deg R(X)\le 2mn}$. Lemma~\ref{lsilv} and estimate~\eqref{esumheal} imply that 
$$
\sum_{\alpha\in \AA}\partial_{\K(\alpha)/\K} \le 4mn\sum_{\alpha\in \AA}\aheight(\alpha) +2mn\log(2mn) \le 8mn^2\pheight(f)+ 14mn^2\log(4mn). 
$$
The field $\K(\alpha)(Q)$ is contained in the field generated over~$\K(\alpha)$ by the coefficients of the Puiseux expansions of~$y$ at~$Q$. Using item~\ref{isumli} of Proposition~\ref{peisen}, but with polynomial\footnote{Recall that ${f^{(\alpha)}(X,Y)=f(X+\alpha, Y)}$.}~$f^{(\alpha)}$ instead of~$f$, we obtain 
$$
 \sum_{x(Q)=\alpha}\partial_{\K(\alpha)(Q)/\K(\alpha)}\le 8n \bigl(\ord_\alpha D(X)+1\bigr)\bigl(\pheight(f^{(\alpha)}) + 5n +\log m\bigr).
$$
Hence, applying~\eqref{esumfal}, we obtain%
\begin{align*}
\sum_{\alpha\in \AA} \sum_{x(Q)=\alpha}\partial_{\K(\alpha)(Q)/\K(\alpha)} &\le 8n\left( \sum_{\alpha\in \AA}\ord_\alpha R(X)\pheight(f^{(\alpha)}) +\sum_{\alpha\in \AA}\pheight(f^{(\alpha)})\right)\\
&\hphantom{\le}+ 8n\left(\sum_{\alpha\in \AA}\ord_\alpha R(X) +|\AA|\right)\bigl( 5n +\log m\bigr)\\
&\le 64mn^2\pheight(f)+144m^2n^2+208mn^3\bigr.
\end{align*}
Finally, Lemma~\ref{lsilv} implies that 
$$
\sum_{r=1}^n\partial_{\K(\zeta_r)/\K} \le \sum_{r=1}^n \log r \le n\log n. 
$$
Combining all this, we obtain~\eqref{eparlk}. \qed

\section{A Tower of $\bark$-Points}
\label{stower}

In this section we retain the set-up of Section~\ref{sprox}; that is, we fix a number field~$\K$, a curve~$\CC$ defined over~$\K$ and rational functions ${x,y\in \K(\CC)}$ such that ${\K(\CC)=\K(x,y)}$. Again, let ${f(X,Y)\in \K[X,Y]}$ be the $\K$-irreducible polynomial of $X$-degree~$m$ and $Y$-degree~$n$ such that ${f(x,y)=0}$, and we again assume that $f_0(X)$ in~\eqref{efxy} is monic. We again define  the polynomial $R(X)$, the sets  ${\AA\subset \bark}$, ${\QQ\subset\CC(\bark)}$ and ${T_1,\ldots, T_6\subset M_\K}$, etc. 

We also fix a covering ${\tilcc \stackrel\phi\to\CC}$ of~$\CC$ by another smooth irreducible projective curve~$\tilcc$; we assume that both~$\tilcc$ and the covering~$\phi$ are defined over~$\K$. We consider $\K(\CC)$ as a subfield of $\K(\tilcc)$; in particular, we identify the functions ${x\in \K(\CC)}$ and ${x\circ\phi\in \K(\tilcc)}$.  We fix a function ${\tily \in \K(\tilcc)}$ such that ${K(\tilcc)=\K(x,\tily )}$. We let ${\tilf(X,\tilY ) \in \K[X,\tilY ]}$ be an irreducible polynomial of $X$-degree~$\tilm$ and $\tilY $-degree~$\tiln$ such that ${\tilf(x,\tily )=0}$; we write 
\begin{equation*}
\tilf(X,\tilY )= \tilf_0(X)\tilY ^\tiln+\tilf_1(X)\tilY ^{\tiln-1}+ \cdots+ \tilf_\tiln(X)
\end{equation*}
and assume that the polynomial $\tilf_0(X)$ is monic. We define in the similar way the polynomial $\tilr(X)$, the sets  ${\tilaa\subset \bark}$, ${\tilqq\subset\tilcc(\bark)}$ and ${\tilt_1,\ldots, \tilt_6\subset M_\K}$, etc. For defining~$\tilt_5$ and~$\tilt_6$ we need to assume that 
\begin{gather}
\label{eallintil}
\tilqq\subset \tilcc (\K), \\
\label{erootsintil}
\text{$\K$ contains $e_\tilq$-th roots of unity for all ${\tilq\in \tilqq}$}. 
\end{gather}

 We also define the notion of proximity on the curve~$\tilcc$ exactly in the same way as we did it for~$\CC$ in Definition~\ref{dprox}, and we have the analogues of Propositions~\ref{pclose},~\ref{pram} and~\ref{pquant}.

In addition to all this, we define one more finite set of places of the field~$\K$ as follows. Write
${\tilr(X)=\tilr_1(X)\tilr_2(X)}$, where the polynomials ${\tilr_1(X),\tilr_2(X)\in \K(X)}$ are uniquely defined by the following conditions:

\begin{itemize}
\item
the roots of $\tilr_1(X)$ are contained in the set of the roots of $f_0(X)$;

\item
the polynomial $\tilr_2(X)$ has no common roots with $f_0(X)$ and is monic.
\end{itemize}
Now let~$\Theta$ be the resultant of $f_0(X)$ and $\tilr_2(X)$. Then ${\Theta \ne 0}$ by the definition of $\tilr_2(X)$, and we set
$$
U=\{v\in M_\K : |\Theta|_v<1\}. 
$$

\begin{proposition}
\label{ptower}
Assume~\eqref{eallin},~\eqref{erootsin},~\eqref{eallintil} and~\eqref{erootsintil}. 
Let ${P\in \CC(\bark)}$ be semi-defined over~$\K$ (that is, ${\xi=x(P) \in \K}$), and let ${\tilp\in \tilcc(\bark)}$ be a point above~$P$ (that is, ${\phi(\tilp)=P}$). Let~$v$ be a finite place of~$\K$, and~$\barv$ an extension of~$v$ to~$\bark$. Assume that~$\tilp$ is $\barv$-close to some ${\tilq\in \tilqq}$. Then we have one of the following options.

\begin{itemize}

\item
$|\xi|_{v} > 1$.

\item
$v\in T\cup\tilt\cup U$.

\item
$P$ is $\barv$-adically close to the ${Q\in \CC(\bark)}$ which lies below~$\tilq$.

\end{itemize}

\end{proposition}

For the proof we shall need a simple lemma.

\begin{lemma}
\label{lyz}
In the above set-up, there exists a polynomial ${\Phi(X,\tilY )\in \K[X,\tilY ]}$ such that 
$$
y=\frac{\Phi(x,\tily )}{f_0(x)\tilr(x)}
$$
\end{lemma}

\begin{proof}
Since $f_0(x)y$ is integral over $\K[x]$, Corollary~\ref{cint} implies that ${f_0(x)y\in \tilr(x)^{-1}\K[x,\tily ]}$,  whence the result. \qed
\end{proof}

\paragraph{Proof of Proposition~\ref{ptower}}
We put  ${\alpha= x(\tilq)}$. By the definition of the set~$\tilqq$, we have ${\alpha \in \tilaa}$. Assume that ${|\xi|_v\le 1}$ and  ${v\notin T\cup\tilt\cup U}$. Let~$\tile$ be the ramification of~$\tilq$ over~$\PP^1$, and let 
\begin{equation}
\label{ezqtil}
\tily ^{(\tilq)}_i= \sum_{k = -k^{(\tilq)}}^{\infty} a_k^{(\tilq)}\tilzeta^{(j-1)k}(x-\alpha)^{k/\tile} \qquad (j=1, \ldots, \tile),
\end{equation}
be the equivalent Puiseux expansions of~$\tily $ at~$\tilq$ (here~$\tilzeta$ is a primitive $\tile$-th root of unity). Since~$\tilp$ is $\barv$-close to~$\tilq$, we have ${|\xi-\alpha|_\barv<1}$ and the~$\tile$ series~(\ref{ezqtil}) converge at~$\xi$, with one of the sums being $\tily (\tilp)$. 

Now let  ${\Phi(X,\tilY )}$ be the polynomial from Lemma~\ref{lyz}. Then the $\tile$ series 
\begin{equation}
\label{eyz}
\frac{\Phi\bigl(x,\tily ^{(\tilq)}_j\bigr)}{f_0(x)\tilr(x)} \qquad (j=1, \ldots, \tile)
\end{equation}
contain all the equivalent  Puiseux series of~$y$ at ${Q=\phi(\tilq)}$. More precisely, if the ramification of~$Q$ over~$\PP^1$ is~$e$, then every of the latter series occurs in~(\ref{eyz}) exactly ${\tile/e}$ times. 

Write ${f_0(X)\tilr(X)=(X-\alpha)^r g(X)}$ with ${g(\alpha)\ne 0}$. 
The assumption
${v\notin T_2\cup\tilt_2\cup\tilt_3\cup\tilt_4\cup U}$ implies that ${|g(\alpha)|_\barv=1}$. Now Lemma~\ref{loneoverf} implies that the Laurent series at~$\alpha$ of the rational function ${\bigl(f_0(x)\tilr(x)\bigr)^{-1}}$ converges at~$\xi$. Hence all the series~(\ref{eyz}) converge at~$\xi$, and among the sums we find 
$$
\frac{\Phi\bigl(x(\tilp),\tily (\tilp)\bigr)}{f_0\bigl(x(\tilp)\bigr)\tilr\bigl(x(\tilp)\bigr)}= y(P). 
$$
Hence~$P$ is $\barv$-close to~$Q$.  \qed

\medskip

We shall also need a bound for~$U$ similar to that of Proposition~\ref{pquant}. 

\begin{proposition}
\label{puquant}
We have ${\vheight(U) \le\Upsilon +\Xi}$, where~$\Upsilon$ is defined in~(\ref{eom}) and
\begin{equation}
\label{exi}
\Xi = 2m\tiln(2\tilm+3\log\tiln) + (m+2\tilm\tiln)\log (m+2\tilm\tiln).
\end{equation}
\end{proposition}

\begin{proof}
Item~\ref{idenum} of Proposition~\ref{pvhei} implies that ${\vheight (U)\le \aheight(\Theta)}$, where~$\Theta$ is the resultant of $f_0(X)$ and $\tilr_2(X)$. Expressing~$\Theta$ as the familiar determinant, we find
\begin{equation}
\label{eahtheta}
\aheight(\Theta) \le \deg\tilr_2 \aheight(f_0)+ \deg f_0\aheight(\tilr_2) + (\deg f_0+\deg\tilr_2) \log (\deg f_0+\deg\tilr_2).
\end{equation}
Since both~$f_0$ and~$\tilr_2$ are monic polynomials (by the convention~\eqref{ef0monic} and the definition of~$\tilr_2$), we may replace the affine heights by the projective heights. Further, we have the estimates
\begin{gather*}
\deg f_0\le m, \qquad \deg\tilr_2 \le \tilm (2\tiln-1), \qquad \pheight(f_0) \le \pheight(f), \\
\pheight(\tilr_2) \le (2\tiln-1)\pheight(\tilf)+ (2\tiln-1)\left(2\tilm+\log\bigl((\tiln+1)\sqrt \tiln\bigr)\right),
\end{gather*}
the latter estimate being a consequence of Corollary~\ref{cdivi} and Lemma~\ref{lheires}. 
Substituting all this to~\eqref{eahtheta}, we obtain the result.
\qed 
\end{proof}

\section{The Chevalley-Weil Theorem}
\label{sproofcw}

Now we may to gather the fruits of our hard work. In this section we retain the set-up of Section~\ref{stower}. Here is our principal result, which will easily imply all the theorems stated in the introduction. 

\begin{theorem}
\label{tmain}
Assume~\eqref{eallin},~\eqref{erootsin},~\eqref{eallintil} and~\eqref{erootsintil}. 
Assume that the covering~$\phi$ is unramified outside the poles of~$x$. Let ${P\in \CC(\bark)}$ be semi-defined over~$\K$, and let ${\tilp\in \tilcc(\bark)}$ be a point above~$P$. As before, we put ${\xi=x(P)=x(\tilp)}$. 
Then for every  ${v\in M_\K^0}$ we have one of the following options.
\begin{itemize}

\item
$|\xi|_{v} > 1$.
\item
$v\in T\cup\tilt\cup U$.

\item
Any extension of~$v$ to $\K(P)$ is unramified in $\K(\tilp)$. 

\end{itemize}
\end{theorem}

\begin{proof}
Let ${v\in M_\K}$ be a non-archimedean valuation such that ${|\xi|_{v} \le  1}$ and ${v\notin T\cup\tilt\cup U}$. 
Fix an extension~$\barv$ of~$v$ to $\bark$, and let~$\tilw$ and $w$ be the restrictions of~$\barv$ to ${\K(\tilp)}$ and ${\K(P)}$, 
and~$\tile$ and~$e$ their ramification indexes  over~$v$, respectively. We want to show that ${\tile=e}$. 

We may assume that  ${\tilp\ \notin \tilqq}$; otherwise there is nothing to prove by~\eqref{eallintil}.  Proposition~\ref{pclose} applied to the covering ${\tilcc\to \PP^1}$ implies that either ${\tile=1}$ and we are done, or  $\tilp$~is $\barv$-adically close to some $\tilq\in\tilqq$, which will be assumed in the sequel. Now Proposition~\ref{pram} implies  that ${\tile=e_\tilq/\gcd(e_\tilq,\ell)}$. Let~$Q$ be the point of~$\CC$ lying under~$\tilq$. Put ${\alpha =x(\tilq)=x(Q)}$.  If ${\alpha \not \in \AA}$ then the covering ${\CC\mapsto\PP^1}$ does not ramify at~$Q$. Since~$\phi$ is unramified outside the poles of~$x$, the covering ${\tilcc\mapsto\PP^1}$ does not ramify at~$\tilq$, that is, ${e_\tilq=1}$. 
Hence ${\tile=1}$, which means that~$v$ is not ramified in $\K(\tilp)$. 

Now assume that ${\alpha\in\AA}$.  Proposition~\ref{ptower} implies that~$P$ is $\barv$-adically close to~$Q$. Now notice that ${e_Q=e_\tilq}$, again because~$\phi$ is unramified. Also, ${\ell(P,Q,v)=\ell(\tilp,\tilq,v)=\ell}$, just by the definition of this quantity. 
Again using Proposition~\ref{pram}, we obtain  that ${e=e_Q/\gcd(e_Q,\ell) = \tile}$. 
This shows tha~$\tilw$ is unramified over~$w$, completing the proof.\qed
\end{proof}

\medskip

We also need an estimate for ${\vheight( T\cup\tilt\cup U)}$. 
Recall the notation
\begin{equation*}
\begin{gathered}
\Omega= mn^2 \bigl(\pheight(f)+2m+2n\bigr),\qquad
\tilom= \tilm\tiln^2 \bigl(\pheight(\tilf)+2\tilm+2\tiln\bigr),\\
\Upsilon=2\tiln\bigl(\tilm\pheight(f)+m\pheight(\tilf)\bigr). 
\end{gathered}
\end{equation*}

\begin{proposition}
\label{phttu}
Assume~\eqref{eallin},~\eqref{erootsin},~\eqref{eallintil} and~\eqref{erootsintil}, and assume in addition that 
\begin{gather}
\label{eunramf}
\text{the covering ${\CC\stackrel x\to \PP^1}$ does not ramify over the roots of $f_0(X)$},\\
\label{eunramfti}
\text{the covering ${\tilcc\stackrel x\to \PP^1}$ does not ramify over the roots of $\tilf_0(X)$}.
\end{gather} 
Then
\begin{equation}
\label{ettu}
\vheight( T\cup\tilt\cup U)\le 60 (\Omega+\tilom)+\Upsilon.
\end{equation}
\end{proposition}

\begin{proof} 
Combining Propositions~\ref{pquant} and~\ref{puquant}, we obtain the estimate
$$
\vheight( T\cup\tilt\cup U)\le 52(\Omega+\tilom)+\Upsilon+\Xi,
$$
where~$\Xi$ is defined in~\eqref{exi}. A routine calculation show that 
${\Xi\le 6(\Omega+\tilom)}$, which proves~\eqref{ettu}.  \qed
\end{proof}

\medskip

Now we can prove the theorems from the introduction.

\paragraph{Proof of Theorem~\ref{tproj}}
We may replace~$\K$ by~$\K(P)$ and assume that ${P\in \CC(\K)}$. Put ${\xi=x(P)}$.

Assume first that~\eqref{eallin},~\eqref{erootsin},~\eqref{eallintil} and~\eqref{erootsintil} hold, and assume in addition that\footnote{We have to replace here~\eqref{eunramf} and~\eqref{eunramfti} by more restrictive conditions~\eqref{eunramff} and~\eqref{eunramftifti} because in the proof we deal not only with the function~$x$, but with~$x^{-1}$ as well.} 
\begin{gather}
\label{eunramff}
\text{the covering ${\CC\stackrel x\to \PP^1}$ does not ramify over the roots of $f_0(X)X^mf_0(X^{-1})$}, \\
\label{eunramftifti}
\text{the covering ${\tilcc\stackrel x\to \PP^1}$ does not ramify over the roots of $\tilf_0(X)X^\tilm\tilf_0(X^{-1})$}.
\end{gather} 
 Theorem~\ref{tmain} and estimate~\eqref{ettu} imply that 
$$
\vheight\bigl(\{v \in \ram(\K(\tilP)/\K): |\xi|_v\le 1\}\bigr) \le 60 (\Omega+\tilom)+\Upsilon. 
$$
Replacing~$x$ by~$x^{-1}$ and the polynomials~$f$,~$\tilf$ by ${X^mf(X^{-1},Y)}$ and ${X^\tilm\tilf(X^{-1},Y)}$, respectively, we obtain the estimate 
$$
\vheight\bigl(\{v \in \ram(\K(\tilP)/\K): |\xi|_v\ge 1\}\bigr) \le 60 (\Omega+\tilom)+\Upsilon. 
$$
Thus, 
$$
\vheight\bigl(\ram(\K(\tilP)/\K)\bigr) \le 120 (\Omega+\tilom)+2\Upsilon, 
$$
and  Lemma~\ref{ldehen} implies that 
\begin{equation}
\label{eassuall}
\partial_{\K(\tilp)/\K} \le \frac{\nu-1}\nu\bigl(120 (\Omega+\tilom)+2\Upsilon\bigr)+1.26\nu\le 120 (\Omega+\tilom)+2\Upsilon. 
\end{equation}

Now let us relax our assumptions. Suppose that we no longer assume~\eqref{eallin},~\eqref{erootsin},~\eqref{eallintil} and~\eqref{erootsintil}, but continue to assume~\eqref{eunramff} and~\eqref{eunramftifti}.  Then~\eqref{eassuall} should be replaced by
\begin{equation}
\label{eassufew}
\partial_{\L(\tilp)/\L} \le 120 (\Omega+\tilom)+2\Upsilon, 
\end{equation}
where~$\L$ is the compositum of the fields $\K(Q)$, $\K(\tilq)$  and the fields generated over~$\K$ by $e_Q$-th and $e_\tilq$-th roots of unity, for all ${Q\in\QQ}$ and ${\tilq\in\tilqq}$.
Proposition~\ref{pcompositum} implies that 
${\partial_{\L/\K}\le 110(\Omega+\tilom)}$. Hence 
$$
\partial_{\K(\tilp)/\K}\le \partial_{\L(\tilp)/\K}\le \partial_{\L(\tilp)/\L}+ \partial_{\L/\K}\le 230(\Omega+\tilom) +2\Upsilon. 
$$

Finally, suppose that we no longer assume~\eqref{eunramff} and~\eqref{eunramftifti} either. All finite ramification points are contained in the set~$\AA$. Hence there is at most ${|\AA|\le (2n-1)m}$ finite ramification points. It follows that there exists a root of unity~$\zeta$ of order $4m^2n$ such that ${f(X, \zeta)X^mf(X^{-1},\zeta)\vert_{X=\alpha}\ne 0}$ for any finite ramification point~$\alpha$. Now instead of the function~$y$ we consider the new function ${z=(y-\zeta)^{-1}\in \bar\K(\CC)}$. It satisfies the equation ${g(x,z)=0}$, where the polynomial 
$$
g(X,Z)=Z^nf(X,\zeta+Z^{-1})=g_0(X)Z^n+g_1(X)Z^{n-1}+\cdots+g_n(X) \in \K(\zeta)[X,Z]
$$
satisfies 
\begin{equation}
\label{enewg}
\deg_Xg=m, \quad \deg_Zg=n, \quad \pheight(g) \le \pheight(f)+ 2n\log 2
\end{equation}
(we use Corollary~\ref{cfalpha}). Also,
${\partial_{\K(\zeta)/\K} \le \log(4m^2n)}$ by Lemma~\ref{lsilv}. 

We have ${g_0(X)X^mg_0(X^{-1})=f(X,\zeta)X^mf(X^{-1},\zeta)}$, and by the choice of~$\zeta$ the covering ${\CC\stackrel x\to \PP^1}$ does not ramify over the roots of ${g_0(X)X^mg_0(X^{-1})}$. 

In the same way we find a root of unity~$\tilzeta$ of order ${4\tilm^2\tiln}$ such that the function ${\tilz=(\tily-\tilzeta)^{-1}}$ satisfies ${\tilg(x,\tilz)=0}$ with ${g(X,\tilZ)\in \K(\tilzeta)[X,\tilZ]}$ satisfying 
\begin{equation}
\label{enewtilg}
\deg_X\tilg=\tilm, \quad \deg_\tilZ\tilg=\tiln, \quad \pheight(\tilg) \le \pheight(\tilf)+ 2\tiln\log 2
\end{equation}
and the covering ${\tilcc\stackrel x\to \PP^1}$ is unramified over the roots of the polynomial ${\tilg_0(X)X^\tilm\tilg_0(X^{-1})}$. Also, ${\partial_{\K(\tilzeta)/\K} \le \log(4\tilm^2\tiln)}$.

Thus,~\eqref{eunramff} and~\eqref{eunramftifti} hold with~$f$,~$\tilf$ replaced by~$g$,~$\tilg$. It follows that 
$$
\partial_{\K(\zeta,\tilzeta)(\tilp)/\K(\zeta,\tilzeta)}\le  230(\Omega'+\tilom') +2\Upsilon', 
$$
where~$\Omega'$,~$\tilom'$ and~$\Upsilon'$ are defined like~$\Omega$,~$\tilom$ and~$\Upsilon$ but with~$f$,~$\tilf$ replaced by~$g$,~$\tilg$. Hence 
\begin{align}
\partial_{\K(\tilp)/\K}&\le  230(\Omega'+\tilom') +2\Upsilon'+\partial_{\K(\zeta)/\K}+\partial_{\K(\tilzeta)/\K}\nonumber\\
\label{ealmost} 
&\le  230(\Omega'+\tilom') +2\Upsilon'+\log(4m^2n)+ \log(4\tilm^2\tiln). 
\end{align}
A messy calculation using~\eqref{enewg} and~\eqref{enewtilg}  shows that the right-hand side of~\eqref{ealmost} does not exceed ${400(\Omega+\tilom)+2\Upsilon+ 6m\tiln^2}$. Theorem~\ref{tproj} is proved.  \qed

\paragraph{Proof of Theorem~\ref{taff}}
Let~$S'$ be set of places of the field~$\K(P)$ extending the places from~$S$. The right-hand side of~\eqref{eaff} will not increase (see item~\ref{isl} of Proposition~\ref{pvhei})  if we replace~$\K$ by~$\K(P)$ and~$S$ by~$S'$. Thus, we may assume that ${P\in \CC(\K)}$. As in the proof of Theorem~\ref{tproj} assume first that~\eqref{eallin},~\eqref{erootsin},~\eqref{eallintil} and~\eqref{erootsintil} hold, and in addition assume\footnote{In this proof we deal only with the function~$x$, and do not need $x^{-1}$, as we did in the projective case. Therefore we may assume~\eqref{eunramf} and~\eqref{eunramfti}, and do not need more restrictive~\eqref{eunramff} and~\eqref{eunramftifti}.}~\eqref{eunramf} and~\eqref{eunramfti}. Again using Theorem~\ref{tmain} and~\eqref{ettu}, we obtain 
$$
\vheight\bigl(\ram(\K(\tilp)/\K)\smallsetminus S \bigr) \le 60 (\Omega+\tilom)+\Upsilon,
$$
and applying Lemma~\ref{ldehen}, we obtain
$$
\partial_{\K(\tilp)/\K} \le 60 (\Omega+\tilom)+\Upsilon+\vheight(S). 
$$
Now we get rid of the assumptions~\eqref{eallin},~\eqref{erootsin},~\eqref{eallintil},~\eqref{erootsintil},~\eqref{eunramf} and~\eqref{eunramfti} in exactly the same manner as we did in the proof of Theorem~\ref{tproj}. The details are routine, we leave them out. 
\qed 

\medskip

To prove Theorem~\ref{tminim}, we need the following result from~\cite{BS10}. 

\begin{theorem}
\label{ter}
Let ${x:\CC\to\PP^1}$ be a finite covering of degree ${n\ge 2}$, defined over~$\K$ and unramified outside a finite set ${A\subset \PP^1(\bark)}$. Put ${h=\aheight(A)}$ and 
${\Lambda'=\bigl(2(\genus+1)n^2\bigr)^{10\genus n+12n}}$,
where ${\genus=\genus(\CC)}$. 
Then there exists  a rational function ${y\in \bark(\CC)}$ such that ${\bark(\CC) =\bark(x,y)}$ and the  rational functions ${x,y\in \bark(\CC)}$ satisfy the equation ${f(x,y)=0}$, where ${f(X,Y)\in \bark[X,Y]}$ is an absolutely irreducible polynomial satisfying
\begin{equation}
\deg_Xf =\genus+1, \qquad \deg_Yf=n, \qquad \pheight(f) \le \Lambda' (h+1).
\end{equation}
Moreover, the number field~$\L$, generated over~$\K$ by the set~$A$ and by the coefficients of~$f$ satisfies
${\partial_{\L/\K(A)}\le \Lambda' (h+1)}$. 
\end{theorem}

\paragraph{Proof of Theorem~\ref{tminim}}
We shall prove the ``projective'' case (that is, item~\ref{iproj}) of this theorem. The ``affine'' case is proved similarly.

We define~$\tilam'$ in the same way as~$\Lambda'$ in Theorem~\ref{ter}, but with~$n$ and~$\genus$ replaced by~$\tiln$ and~$\tilgen$. 
We use Theorem~\ref{ter} to find functions ${y\in \bark(\CC)}$ and ${\tily \in \bark(\tilcc)}$, and polynomials ${f(X,Y) \in \bark[X,Y]}$ and ${\tilf(X,\tilY ) \in \bark[X,\tilY ]}$. Denoting by~$\L$ the field generated by the set~$A$ and the coefficients of both the polynomials, we find 
${\partial_{\L/\K(A)}\le (\Lambda'+\tilam') (h+1)}$ with ${h=\aheight(A)}$. 
Using Lemma~\ref{lsilv}, we estimate 
${\partial_{\K(A)/\K} \le 2(\delta-1)h+\log\delta}$. Hence
$$
\partial_{\L/\K}\le \bigl(\Lambda'+\tilam'+2(\delta-1)\bigr) (h+1). 
$$
We define the quantities~$\Omega$,~$\tilom$ and~$\Upsilon$ as in the introduction. Then, applying Theorem~\ref{tproj}, but over the field~$\L$ rather than~$\K$, we find
${\partial_{\L(\tilp)/\L(P)}\le 400(\Omega+\tilom)+2\Upsilon+6m\tiln^2}$. We have
$$
\partial_{\K(\tilp)/\K(P)} \le \partial_{\L(\tilp)/\K(P)} = \partial_{\L(\tilp)/\L(P)}+\partial_{\L(P)/\K(P)} \le
\partial_{\L(\tilp)/\L(P)}+\partial_{\L/\K}. 
$$
The last sum is bounded by 
$$
400(\Omega+\tilom)+2\Upsilon+ 6m\tiln^2 +\bigl(\Lambda'+\tilam'+2(\delta-1)\bigr) (h+1),
$$
which, obviously, does not exceed ${\Lambda(h+1)}$, as wanted. \qed


\bigskip

\begin{tabular}{lll}
\textbf{Yuri Bilu}&\textbf{Marco Strambi}&\textbf{Andrea Surroca}\\
IMB, Universit\'e Bordeaux 1&via del Litorale 145&Mathematisches Institut\\
351 cours de la Libération&Antignano, Livorno&Universit\"at Basel\\
33405 Talence CEDEX&57128 Italy&Rheinsprung 21\\
France&&CH-4051 Basel\\
&&Switzerland
\end{tabular}

\end{document}